\documentclass[11pt]{article}
\usepackage{amsmath,amsfonts,amsthm,amssymb}
\usepackage[margin=1in,letterpaper]{geometry}
\usepackage{physics}
\usepackage[dvipdfmx]{graphicx}
\usepackage{tikz}
\usetikzlibrary{graphs,graphs.standard}
\usetikzlibrary{positioning,automata}
\usepackage{amsthm}
\theoremstyle{definition}
\newtheorem{defi}{Definition}[section]
\newtheorem{theorem}{Theorem}

\newtheorem{lemma}[defi]{Lemma}
\newtheorem{prop}[defi]{Proposition}

\newtheorem{claim}[defi]{Claim}

\usepackage{mathtools}
\usepackage{bm}

\newcommand{\mi}{\mathrm{i}}
\newcommand{\mj}{\mathrm{j}}
\newcommand{\mk}{\mathrm{k}}
\newcommand{\be}{{\bm e}}
\newcommand{\tri}{\mathrm{tri}}

\DeclareMathOperator{\id}{id}
\DeclareMathOperator{\Hom}{Hom}

\title{Characterizing the Universal Rigidity of Generic Tensegrities}
\author{Ryoshun Oba\thanks{Department of Mathematical Informatics, Graduate School of Information Science and Technology, University of Tokyo, Tokyo 113-8656, Japan. Email: \texttt{ryoshun\_oba@mist.i.u-tokyo.ac.jp}}
 \and Shin-ichi Tanigawa\thanks{Department of Mathematical Informatics, Graduate School of Information Science and Technology, University of Tokyo, Tokyo 113-8656, Japan. Email: \texttt{tanigawa@mist.i.u-tokyo.ac.jp}}}

\begin{document}
\maketitle
\begin{abstract}
    A tensegrity is a structure made from cables, struts and stiff bars. 
    A $d$-dimensional tensegirty is universally rigid  if it is rigid in any dimension $d'$ with $d'\geq d$.
    The celebrated super stability condition due to Connelly gives a sufficient condition for a tensegrity to be universally rigid.
    Gortler and Thurston showed that super stability characterizes universal rigidity when the point configuration is generic and every member is a stiff bar.
    We extend this result in two directions.
    We first show that a generic universally rigid tensegrity is super stable.
    We then extend it to tensegrities with point group symmetry, and show that this characterization still holds  as long as a tensegrity is  generic modulo symmetry.
    Our strategy is based on the block-diagonalization technique for symmetric semidefinite programming problems, and our proof relies on the theory of real irreducible representation of finite groups.
\end{abstract}
\section{Introduction}
A {\em tensegrity} is a stable structure made from cables, struts, and stiff bars. 
Since the invention by Kenneth Snelson, the theory of tensegrities and applications have been extensively studied from various perspectives. 
A mathematical foundation for the rigidity or stability analysis has been established in the context of rigidity theory~\cite{connelly1982rigidity,CW95,RW81}.
Following a notation in that context, we define a {\em ($d$-dimensional)  tensegrity}  as a triple $(G, \sigma, p)$ of an {\em edge-signed graph} $(G,\sigma)$ with $\sigma:E(G)\rightarrow \{-1,0,+1\}$ and a {\em point-configuration} $p:V(G)\rightarrow \mathbb{R}^d$.
Here each vertex $i$ corresponds to a joint $p_i=p(i)\in \mathbb{R}^d$, 
each edge $e=ij$ with $\sigma(e)=+1/0/-1$ corresponds to a cable/bar/strut, respectively, between joints $p_i$ and  $p_j$.
When every member is a stiff bar (that is, $\sigma(e)=0$ for every $e\in E(G)$), a tensegrity is called a {\em bar-joint framework},
which is the central object of study in rigidity theory. 

In a tensegrity all bars are stiff and cannot change the length while cables can be shorter and struts can be longer.
Under the system of these geometric constraints, the global rigidity of the tensegrity is defined in terms of the uniqueness of the solution of the system  up to isometries. More formally, given an edge-signed graph $(G,\sigma)$,  two point-configurations $p, q$ for $(G,\sigma)$ are said to be {\em  congruent} if 
\[
\|p_i -p_j\|=\|q_i -q_j\| \text{ for all } i, j \in V(G),
\]  
where $\|\cdot\|$ denotes the Euclidean norm, and a tensegrity  $(G,\sigma,p)$ is {\em congruent} to a tensegrity $(G,\sigma,q)$ if $p$ is congruent to $q$. 
We say that  a tensegrity  $(G,\sigma,p)$ {\em dominates} a tensegrity $(G,\sigma,q)$ 
if 
\begin{align*}
\|p_i -p_j\|&\geq \|q_i -q_j\| \text{ for all $e=ij \in E(G)$ with $\sigma(e)=+1$}, \\
\|p_i -p_j\|&= \|q_i -q_j\| \text{ for all $e=ij \in E(G)$ with $\sigma(e)=0$, and} \\ 
\|p_i -p_j\|&\leq \|q_i -q_j\| \text{ for all $e=ij \in E(G)$ with $\sigma(e)=-1$}.
\end{align*} 
This dominance captures the set of possible deformations of a given tensegrity $(G,\sigma, p)$, 
where a tensegrity $(G,\sigma,q)$ satisfies the geometric constraints posed by cables/bars/struts of $(G,\sigma, p)$ if and only if $(G,\sigma,q)$ is dominated by $(G,\sigma,p)$.
A $d$-dimensional tensegrity $(G,\sigma,p)$ is {\em globally rigid} if every $d$-dimensional tensegrity $(G,\sigma,q)$ dominated by $(G,\sigma,p)$ is congruent to $(G,\sigma,p)$.

Connelly~\cite{connelly1982rigidity} initiated the rigidity analysis of tensegrities, and in his paper \cite{connelly1982rigidity} in 1982 he gave a celebrated sufficient condition for the global rigidity in terms of stress matrices (that is, graph Laplacian matrices weighted by equilibrium self-stresses).
Tensegrities satisfying his sufficient condition are called {\em super stable}, and super stability is now used as a major criteria for structural engineers to develop new tensegrities (see, e.g., \cite{ZO15}).

Recently Connelly's super stability condition got an attention in the context of the sensor network localization or the graph realization problem~\cite{alfakih2007dimensional,A14,SY07}. 
To understand the exact solvability of the SDP relaxation, Ye and So~\cite{SY07} looked at a stronger rigidity property, called universal rigidity. Suppose that $(G,\sigma,p)$ is a $d$-dimensional tensegrity whose ambient space $\mathbb{R}^d$ lies in $\mathbb{R}^{d'}$ for each integer $d'\geq d$.
Then $(G,\sigma,p)$ is also a tensegrity in $\mathbb{R}^{d'}$. 
We say that $(G,\sigma,p)$ is {\em universally rigid} if $(G,\sigma,p)$ is globally rigid in $\mathbb{R}^{d'}$ for every integer $d'\geq d$.
Clearly, universal rigidity implies global rigidity but the converse implication does not hold in general as indicated in Figure~\ref{fig:1} (See, e.g.,~\cite{C20} for further interaction between two rigidity concepts.)
\begin{figure}
    \centering
    \scalebox{1.5}{
    \begin{tikzpicture}[every node/.style={draw,fill=black,circle,inner sep=1pt,minimum size=2pt}]
        \node[] (1) {};
        \node[below right=0.5cm and 0.2cm of 1] (2) {};
        \node[below left=0.7cm and 0.4cm of 2] (3) {};
        \node[right =1.5cm of 3] (4) {};
        \node[above left=0.7cm and 0.4cm of 4] (5) {};
        \node[above right=0.5cm and 0.2cm of 5] (6) {};
        \foreach \u / \v in {1/2,2/3,3/4,4/5,5/6,6/1,1/3,2/4,3/5,4/6}
        \draw (\u) -- (\v);
    \end{tikzpicture}
    
    \qquad
\begin{tikzpicture}[every node/.style={draw,fill=black,circle,inner sep=1pt,minimum size=2pt}]
    \node[] (1) {};
    \node[below left=0.5cm and 0.7cm of 1] (2) {};
    \node[below right=0.7cm and 0.5cm of 2] (3) {};
    \node[right =1.5cm of 3] (4) {};
    \node[above right=0.7cm and 0.5cm of 4] (5) {};
    \node[above left=0.5cm and 0.7cm of 5] (6) {};
    \foreach \u / \v in {1/2,2/3,3/4,4/5,5/6,6/1,1/3,2/4,3/5,4/6}
    \draw (\u) -- (\v);
\end{tikzpicture}}
    \caption{A framework on the left is globally rigid in $\mathbb{R}^2$, but it is not universally rigid. A framework on the right is universally rigid since it is a Cauchy polygon. These examples show that universal rigidity is not a generic property of a graph.}
    \label{fig:1}
\end{figure}

Although universal rigidity is stronger than global rigidity, super stability still implies universal rigidity as it is implicit in Connelly's original work~\cite{connelly1982rigidity}. It turns out  that super stability even characterizes universal rigidity for almost all bar-joint frameworks.
Specifically we say that a tensegrity (or a bar-joint framework) is {\em generic} if the set of coordinates of the points is algebraically independent over $\mathbb{Q}$. 
In 2014, Gortler and Thurston~\cite{GT} proved that a generic bar-joint framework $(G,p)$ is universally rigid if and only if it is super stable.

The goal of this paper is to extend the Gortler-Thurston characterization in two directions. 
We first extend the result to tensegrities, and show that universal rigidity and super stability coincide for generic tensegrities.
We then extend it to tensegrities with point group symmetry, 
where  a finite point group faithfully acts on the underlying signed graphs and the point-configurations are compatible with this action (see Section~\ref{sec:4} for the formal definition.) 
Note that a priori a tensegrity with point group symmetry is not generic, 
but we shall prove that a characterization still holds as long as  tensegrities are ``generic modulo symmetry''.
(Such a research direction is widely investigated for infinitesimal rigidity, see, e.g.~\cite{MT14,BW17}, and references therein.) 

As given in tensegrity catalogues, most of existing tensegrities exhibit symmetry or are compositions of simple symmetric modules,
and building larger tensegrities based on group symmetry is now a standard technique in structural engineering. The technique was initiated by Connelly and Terrell~\cite{CT95}, where they showed how to simplify the super stability condition via finite group representation theory.
Although their paper focuses on particular instances, the technique is general enough to design a larger class of  symmetric tensegrities~\cite{CB98}.
An implication of our result is that any  universally rigid tensegrities whose point-configurations are generic modulo symmetry can be obtained 
from stress matrices constructed as in the method of Connelly and Terrell.

We should remark that, as shown by Connelly and Gortler~\cite{connelly2015iterative},  the universal rigidity of tensegrities can be characterized by a sequence of dual solutions in the facial reduction procedure due to Borwein and Wolkowicz~\cite{B81}. 
We however believe that the characterization in terms of stress matrices (or  weighted Laplacian) is important toward characterizing the global rigidity of symmetric tensegrities. A characterization of the global rigidity of generic bar-joint frameworks is known in terms of stress matrices~\cite{gortlerhealythurston}. 

Technically our work is closely related to the topic of {\em strict complementarity} in semidefinite programming (SDP) problems, or equivalently to the face exposedness of projections of positive semidefinite cones. Understanding the existence of strict complementary pair of primal and dual solutions is a classical  but still on-going research topic in convex optimization (see, e.g., \cite{dCT1} for a recent result). 
The characterization problem of the universal rigidity  of bar-joint frameworks is known to be equivalent to the existence of strict complementary pairs of primal and dual solutions in the Euclidean matrix completion problem (see Section~\ref{sec:2} for details). 
There are several researchers that answer the characterization problem (or  the existence of strict complementary pair) for special classes of graphs~\cite{ATY13,DPW15,T17} while Gortler-Thurston~\cite{GT} solved the  problem assuming a certain genericity of input entries.  This paper provides  a new  direction based on group symmetry to go beyond generic instances.

Our proof strategy is based on the block-diagonalization technique for symmetric SDP problems. 
Here the general idea is to use the  block-diagonalization of the underlying matrix algebra to decompose SDP instances to smaller pieces,
and the method is successfully used to solve large scaled SDP problems, see, e.g.,~\cite{bachoc2012invariant,gatermann2004symmetry,murota2010numerical}. 
Also, prior researches~\cite{B09,K01} on this technique are motivated from the optimal design of truss structures.
Our technical contribution is to use the block-diagonalization technique to analyze the facial structures of SDP problems rather than for reducing computational cost, and our proof essentially relies on  the theory of {\em real} irreducible representation.


\section{Semidefinite Programming Problem for Universal Rigidity}\label{sec:2}
In this section we shall explain the background materials for analyzing universal rigidity from the view point of semidefinite programming.

Throughout the paper we shall use the following notations.
Let $V$ be a finite set with $|V|=n$ (typically $V=\{1,2,\dots, n\}$).
For a finite set $X$ with $|X|=m$, let $\mathbb{R}^X$ be the $m$-dimensional Euclidean space whose each entry is indexed by each element of $X$.
For $i\in X$, let $\bm{e}_i$ be the unit vector of $\mathbb{R}^X$ whose $i$-th entry is one and all other entries are zero,
and let ${\bm 1}_X=\sum_{i\in X} \bm{e}_i$.
Similarly, let $\mathcal{S}^X$ be the set of all $m \times m$ symmetric matrices whose entries are indexed by the pairs of elements in $X$. Throughout the paper, $\mathcal{S}^X$ is regarded as a Euclidean space by using the trace inner product $\langle \cdot , \cdot \rangle$ defined by $\langle A,B \rangle=\text{tr}AB$.
If $A \in \mathcal{S}^X$ is positive semidefinite, it is denoted as $A \succeq 0$,
and let $\mathcal{S}^X_+=\{ A \in \mathcal{S}^X: A\succeq 0\}$.
%

For a graph $G$, let $N_G(i)$ be the set of all neighbors of $i\in V(G)$ in $G$, and let $\overline{N}_G(i)=N_G(i)\cup \{i\}$.

\subsection{Weighted Laplacian and Configurations}
For the SDP formulation we shall first define the space of Laplacian matrices. 

Given a graph $G=(V,E)$ with edge weight $\omega:E(G)\rightarrow \mathbb{R}$, its {\em Laplacian} $L_{G,\omega}$ is 
defined by 
\[
L_{G,\omega}:=\sum_{e=ij\in E} \omega_{ij} F_{ij},
\]
where $\omega_{ij}=\omega(ij)$ and
\[
F_{ij}:=(\bm{e}_i-\bm{e}_j)(\bm{e}_i-\bm{e}_j)^\top.
\]
It is symmetric and always satisfies $L_{G,\omega}{\bm 1}_V=\bm{0}$.
A weighted Laplacian of the complete graph on $V$ is simply called a {\em Laplacian matrix}. (Equivalently, a symmetric matrix $L$ is Laplacian if $L{\bm 1}_V=0$.)
Let $\mathcal{L}^V$ be the set of all Laplacian matrices. Then $\mathcal{L}^V$ is a linear subspace of $\mathcal{S}^V$ given by
\[
\mathcal{L}^V={\rm span}\{F_{ij}: i,j\in V, i\neq j\},
\]
where $\{F_{ij}: i,j\in V, i\neq j\}$ forms a basis.

Let $J_V={\bm 1}_V {\bm 1}_V^{\top}$.
When $L\succeq 0$, $L\in \mathcal{L}^V$ if and only if  $\langle L, J_V\rangle=0$.
Hence the set of positive semidefinite Laplacian matrices  $\mathcal{L}_+^V$ is given by 
\[
\mathcal{L}_+^V=\{ L \in \mathcal{S}_+^V : \langle L,J_V \rangle=0 \}. 
\]

Let $q:V\rightarrow \mathbb{R}^d$ be a $d$-dimensional point configuration for some positive integer $d$. 
We identify $q$ with a matrix $Q$ of size $d\times n$ whose $i$th column vector is $q_i$. 
We then have $Q^{\top} Q\succeq 0$, 
and $\langle Q^\top Q, J_V \rangle = 0$ holds if and only if the center of gravity of $q(V)$ is the origin, i.e., $\sum_{i\in V} q_i=\bm{0}$.
$Q^\top Q$ is called the {\em Gram matrix} of $q$.
Since the properties we are interested in (such as universal rigidity) are invariant by translations, 
throughout the paper we shall focus on tensegrities whose center of gravity is at the origin.

We denote by $\mathcal{C}_d(V)$ the set of all point configurations $q:V\rightarrow \mathbb{R}^d$ such that $\sum_{i\in V} q_i=\bm{0}$ and $q(V)$ affinely span $\mathbb{R}^d$,
and let $\mathcal{C}(V)=\bigcup_{d\in \mathbb{Z}_{\geq 0}} \mathcal{C}_d(V)$. 
Then we have  
\begin{equation}\label{eq:conf2}
\{L\in \mathcal{L}^V_+: \rank L=d\}=\{Q^{\top} Q: q\in \mathcal{C}_d(V)\}
\end{equation}
and 
\begin{equation}\label{eq:conf}
\mathcal{L}^V_+=\{Q^{\top} Q: q\in \mathcal{C}(V)\}.
\end{equation}

\subsection{SDP Formulation}
Let $(G,\sigma, p)$ be a $d$-dimensional tensegrity.  
Let $E_0=\sigma^{-1}(0)$, $E_+=\sigma^{-1}(+1)$, $E_-=\sigma^{-1}(-1)$.
We consider the following semidefinite programming problem for $(G,\sigma,p)$:
    \[
\begin{array}{llll}
        \text{(P)} &  \text{max.}  & 0 \\
         & \text{s.t.}  & \langle X, F_{ij} \rangle   =  \| p_i-p_j\|^2 &(ij \in E_0) \\
         &              & \langle X, F_{ij} \rangle \leq \| p_i-p_j\|^2 &(ij \in E_+) \\
         &              & \langle X, F_{ij} \rangle \geq \| p_i-p_j\|^2 &(ij \in E_-) \\
         &              & X \in \mathcal{L}^V_+.
    \end{array}
\]
By (\ref{eq:conf}) any feasible $X$ is written as $X=Q^{\top}Q$ for some $q\in \mathcal{C}(V)$.     
Moreover, 
\[
\langle Q^\top Q, F_{ij}\rangle=\langle Q^\top Q, (\bm{e}_i-\bm{e}_j)(\bm{e}_i-\bm{e}_j)^\top \rangle = \|q_i-q_j\|^2
\]
holds, which means that $Q^\top Q$ is feasible if and only if 
    $(G,\sigma,q)$ is dominated by $(G,\sigma,p)$.
It can be also checked that $Q^\top Q=P^{\top} P$ holds if and only if $p$ and $q$ are congruent.
    Therefore, we have the following.
    \begin{prop} \label{prop:uniqueness}
        $(G,\sigma,p)$ is universally rigid if and only if (P) has a unique feasible solution.
    \end{prop}
Since $\mathcal{L}^V_+\subset \mathcal{L}^V$ and $F_{ij}\in \mathcal{L}^V$, we can consider the dual problem of (P) in $\mathcal{L}^V$, that is, 
\[
    \begin{array}{llll}
        \text{(D)} & \text{min.}  & \sum_{ij \in E(G)} \omega_{ij}\|p_i-p_j\|^2 \\
         & \text{s.t.}  & \sum_{ij \in E(G)} \omega_{ij} F_{ij} \succeq 0 \\
         &   & \sigma(ij)\omega_{ij} \geq 0 & (ij\in E(G)).
    \end{array}
\]
By weak duality, the dual optimal value is at least $0$, and it is indeed $0$ as it is attained by $\omega=0$.

If we consider a dual variable $\omega:E(G)\rightarrow \mathbb{R}$ as an edge weight of $G$, 
the first dual constraint is written by $L_{G,\omega}\succeq 0$. 
Moreover, the objective function is equal to $\langle P^\top P,L_{G,\omega} \rangle$. 
Hence $L_{G,\omega}\succeq 0$ implies that $\omega$ is dual optimal if and only if $PL_{G,\omega}=O$.
In terms of $p$, the latter condition becomes
\begin{equation}\label{eq:equi}
\sum_{j \in N_G(i)} \omega_{ij}(p_i-p_j)=\bm{0} \qquad (i\in V(G)).
\end{equation}

The equation (\ref{eq:equi}) is nothing but the {\em equilibrium condition} for structures to be statically rigid, and the equation frequently appears in rigidity theory. In general, for a tensegrity $(G,\sigma, p)$, an edge weight $\omega:E(G)\rightarrow \mathbb{R}$ is said to be an {\em equilibrium stress} if $\omega$ satisfies (\ref{eq:equi}).
Also $\omega$ is said to be {\em proper} if 
\begin{equation}\label{eq:proper}
\sigma(ij)\omega_{ij}\geq 0\qquad (ij\in E(G)).
\end{equation}
We further say that $\omega$ is {\em strictly proper} if (\ref{eq:proper}) holds with strict inequality for every $ij \in E_+ \cup E_-$.
The condition (\ref{eq:proper}) reflects a physical fact that each cable only has a tension while each strut only has a compression (see~\cite{RW81} for more details).  

With this notation, the discussion is summarized as follows.
    \begin{prop} \label{prop dual and ESM}
     An edge weight $\omega:E(G)\rightarrow \mathbb{R}$ is an optimal solution of (D) if and only if it is a proper equilibrium stress of $(G,\sigma,p)$. 
    \end{prop}

\subsection{Facial Structure of $\mathcal{L}_+^V$}
In the next two subsections, we shall provide high level ideas of Connelly's sufficient condition and Gortler-Thurston's characterization 
since our  technical result will be built on these ideas.
The key ingredient in both results are the facial structure of $\mathcal{L}_+^V$.

Let  $C$ be a non-empty convex set in a Euclidean space.
 The dimension of $C$ is the {\em  dimension} of the smallest affine subspace containing $C$ and is denoted as $\dim C$.
 A convex subset $F \subseteq C$ is a {\em  face} if for any $x,y \in C$, $\frac{x+y}{2} \in F$ implies $x,y \in F$.
 For $x \in C$, the smallest face containing $x$ is called the  {\em minimal face} of $x$ and is denoted as $F_C(x)$.
 We say that a hyperplane $H$ {\em exposes} a face $F$ of $C$ if $F=C\cap H$ and $H$ supports $C$ (i.e., $C$ lies on the closed halfspace defined by $H$). A face $F$ is said to be {\em exposed} if there is a hyperplane exposing $F$. 
 To simplify the presentation, we also consider the ambient space as a hyperplane whose normal vector is zero vector.
 Then $C$ itself is always exposed.
 $C$ is called {\em exposed} if every face of $C$ is exposed.
 It is well-known that ${\cal S}_+^V$ (in ${\cal S}^V$) is exposed, 
 but this is not a general property of convex sets.
 Moreover the following properties are known for the facial structure of ${\cal S}^n_+$ (see, e.g., \cite{pataki2000geometry}).
\begin{prop} \label{prop:psdface}
    Let $A \in {\cal S}^n_+$ be a matrix with rank $d$. Then $\dim F_{{\cal S}^n_+} (A) = \binom{d+1}{2}$.

    For $B \in {\cal S}^n$, the hyperplane $\{X \in {\cal S}^n : \langle X,B \rangle=0\}$ exposes $F_{{\cal S}^n_+} (A)$ if and only if $B$ satisfies $\rank A + \rank B =n$, $\langle A,B \rangle=0$, and $B \succeq 0$.
\end{prop}

Now we are interested in the facial structure of $\mathcal{L}_+^V$ in $\mathcal{L}^V$. 
$\mathcal{L}_+^V$ is known to be a face of $\mathcal{S}_+^V$, and one can understand the facial structure of $\mathcal{L}_+^V$ by restricting the ambient matrix space to $\mathcal{L}^V$. 
Based on Proposition~\ref{prop:psdface}, the following properties easily follow.
\begin{prop}\label{prop:realface}
Let $L\in \mathcal{L}_+^V$ be a matrix with rank $d$.
Then $\dim F_{\mathcal{L}_+^V}(L)={d+1 \choose 2}$.

For $M\in \mathcal{L}^V$,  the hyperplane  $\{X\in \mathcal{L}^V: \langle X, M\rangle=0\}$ exposes $F_{\mathcal{L}_+^V}(L)$ if and only if $M$ satisfies
$\rank L+\rank M=|V|-1$, $\langle L, M\rangle=0$, and $M\succeq 0$.
\end{prop}

\subsection{Connelly's Sufficient Condition}
The following is Connelly's super stability condition.
\begin{theorem}[Connelly~\cite{connelly1982rigidity}]\label{thm:connelly}
Let  $(G,\sigma,p)$ be  a $d$-dimensional tensegrity  with $n$ vertices, and suppose that 
\begin{itemize}
\item[(i)]  it has a strictly proper equilibrium stress $\omega$ such that 
$L_{G,\omega}\succeq 0$ and $\rank L_{G,\omega}=n-d-1$, and 
\item[(ii)]  there is no non-zero symmetric matrix $S$ of size $d\times d$ such that 
\[
(p_i-p_j)^{\top} S(p_i-p_j)=0\qquad (ij\in E(G)).
\]
\end{itemize}
Then $(G,\sigma,p)$ is universally rigid.
\end{theorem}
The condition (ii)  is referred to as the {\em conic condition for the edge directions}.

Connelly~\cite{connelly2005generic} pointed out that for a {\em generic} bar-joint framework with at least $d+1$ vertices the conic condition for the edge directions always holds.
 Recent papers~\cite{alfakih2013affine,connellygortlertheran} examine how to ensure the conic condition for the edge directions without genericity. 
 For practical purpose the following statement due to Alfakih and Nguyen would be sufficiently general.
 \begin{theorem}[Alfakih and Nguyen~\cite{alfakih2013affine}] \label{thm:alfakih}
Let $(G,\sigma,p)$ be a $d$-dimensional tensegrity such that $p(\overline{N}_G(i))$ affinely spans $\mathbb{R}^d$ for each $i\in V(G)$.
Suppose also that $(G,\sigma,p)$ has a strictly proper equilibrium stress $\omega$ such that 
$L_{G,\omega}\succeq 0$ and $\rank L_{G,\omega}=n-d-1$.
Then the conic condition for the edge direction holds.
\end{theorem}
See~\cite{connellygortlertheran} for stronger sufficient conditions.

\subsection{Characterization by Gortler and Thurston}
Gortler-Thuston~\cite{GT} gave a reverse direction of Theorem~\ref{thm:connelly} for generic bar-joint frameworks.
\begin{theorem}[Gortler-Thurston~\cite{GT}]\label{thm:gortler_thurston}
A generic $d$-dimensional bar-joint framework $(G,p)$ with $n\geq d+2$ vertices is universally rigid 
if and only if it has an equilibrium stress $\omega$ such that 
$L_{G,\omega}\succeq 0$ and $\rank L_{G,\omega}=n-d-1$.
\end{theorem}

The sufficiency is due to Connelly. 
The proof of the necessity goes as follows.
Suppose that $(G,p)$ is a universally rigid generic framework, and we want to find  $L_{G,\omega}$ as given in the statement.
Translate $p$ so that the center of gravity is at the origin, and consider the face $F_{\mathcal{L}_+^V}(P^{\top}P)$. 
By Proposition~\ref{prop:realface}, $F_{\mathcal{L}_+^V}(P^{\top}P)$ is exposed by the hyperplane $\{X\in \mathcal{L}^V: \langle X, L\rangle=0\}$ for some $L\in \mathcal{L}^V_+$ with 
\begin{align}
&\rank P^{\top}P+\rank L=n-1, \label{eq:GT1}\\
&\langle P^{\top} P, L\rangle=0. \label{eq:GT2} 
\end{align}
 By $L\in \mathcal{L}^V_+$, we have $L\succeq 0$, and by (\ref{eq:GT1}), we also have $\rank L=n-d-1$.
 Hence, if $L=L_{G,\omega}$ for some $\omega:E(G)\rightarrow \mathbb{R}$ (i.e., $(i,j)$-th entry of $L$ is zero if $ij\notin E(G)$), 
 then Proposition~\ref{prop dual and ESM} and  (\ref{eq:GT2})  imply that $\omega$ is an equilibrium stress. 
Therefore, what is remaining is to prove that  $F_{\mathcal{L}_+^V}(P^{\top}P)$ is exposed by the hyperplane defined by $L_{G,\omega}$ for some $\omega$. 

To find such $L_{G,\omega}$, consider the subspace 
\[
\mathcal{L}(G):={\rm span}\{F_{ij}: ij\in E(G)\}
\]
of $\mathcal{L}^V$. 
The idea is to look at the projection $\pi$ of $\mathcal{L}^V$ to $\mathcal{L}(G)$,
and compute the hyperplane $H$ of $\mathcal{L}(G)$ exposing the minimal face of  $\pi(P^{\top}P)$ in $\pi(\mathcal{L}^V_+)$.
Then  this hyperplane $H$  is defined by $L\in \mathcal{L}(G)$, which is equivalent to having an expression $L=L_{G,\omega}$ for some $\omega$,
and $\pi^{-1}(H)$ (defined by $L$) would be the hyperplane of $\mathcal{L}^V$ exposing $F_{\mathcal{L}_+^V}(P^{\top}P)$  as required.

There is one technical subtlety in this argument: 
the minimal face of  $\pi(P^{\top}P)$ in $\pi(\mathcal{L}_+^V)$ may not be exposed. 
(Even if $\mathcal{L}_+^V$ is exposed, $\pi(\mathcal{L}_+^V)$ may not be exposed.)
The main technical observation of Gortler-Thurston~\cite{GT} is to prove that, if $P^{\top} P$ is generic in a certain sense, $\pi(P^{\top}P)$ is exposed.  

Our proof follows the same technique, and a detailed description will be given in Section~\ref{sec:tensegrity}. 
In order to give a rigorous discussion,  we review the following materials  from  \cite{GT}.
 
  A subset of a Euclidean space is {\em semi-algebraic} over $\mathbb{Q}$ if it is described by finite number of algebraic equalities and inequalities
        whose coefficients are rationals.
        Let $S$ be a semi-algebraic set defined over $\mathbb{Q}$.
        A point $x \in S$ is {\em generic in} $S$ if there is no rational coefficient polynomial $f$
        such that $f(x)=0$ and $f(y) \neq 0$ for some $y \in S$.
        A point $x \in S$ is {\em locally generic in} $S$ if for small enough $\epsilon>0$, $x$ is generic in $S \cap B_\epsilon(x)$.
    The following proposition can be used to "transfer" the genericity of a point configuration  $p$ to  $P^{\top}P$. 
    \begin{prop} [{Gortler-Thurston~\cite[Lemma 2.6]{GT}}]\label{prop inherit genericity}
        Let $S$ be a semi-algebraic set define over $\mathbb{Q}$ and $f$ be an algebraic map from $S$ to a Euclidean space.
        If $x$ is generic in $S$, $f(x)$ is generic in $f(S)$. 
    \end{prop}


Let $C$ be a non-empty convex set in a Euclidean space. 
$C$ is {\em line-free} if it contains no complete affine line.
A point  $x \in C$ is {\em $k$-extreme} if $\dim F_C(x) \leq k$.
We denote by $\text{ext}_k(C)$ the set of $k$-extreme points of $C$.

We will use the following combination of \cite[Proposition 4.14]{GT} and \cite[Theorem 2]{GT}, which is also explicit in the proof of the main theorem of \cite{GT}.
 \begin{prop} \label{prop:GT}
        Let $C$ be a closed line-free convex semi-algebraic set in $\mathbb{R}^m$, and $\pi:\mathbb{R}^m\rightarrow \mathbb{R}^n$ be a projection, both defined over $\mathbb{Q}$.
       Suppose that $x$ is locally generic in $\text{ext}_k(C)$ for some $k$ and $\pi^{-1}(\pi(x))\cap C$ is a singleton set.
        Then there exists a hyperplane $H$ in $\mathbb{R}^n$ such that $\pi^{-1}(H)$ exposes $F_C(x)$.
    \end{prop}   

\section{Characterizing the Universal Rigidity of Tensegrities}\label{sec:tensegrity}
In this section we prove an extension of Theorem~\ref{thm:gortler_thurston} to tensegrities.
 \begin{theorem} \label{maintheoremtens}
        Let $(G,\sigma,p)$ be a generic $d$-dimensional tensegrity with $n \geq d+2$ vertices.
        Then  $(G,\sigma,p)$ is universally rigid if and only if  it has a strictly proper equilibrium stress $\omega$ such that 
        $\rank L_{G,\omega}=n-d-1$ and $L_{G,\omega}\succeq 0$.
 \end{theorem}  
 The sufficiency follows from Theorem~\ref{thm:connelly}, and we focus on the necessity.
 When applying Gortler-Thurston's  proof  to tensegrities, we need to further ensure that $\omega$ is proper, i.e., the sign condition~(\ref{eq:proper}). 
 In the above proof sketch, this requires to find an exposing hyperplane of $F_{\mathcal{L}_+^V}(P^{\top}P)$ satisfying a sign condition on non-zero entries. We show how to get around this  by a simple trick.

\begin{proof}[Proof of Theorem~\ref{maintheoremtens}]
To see the necessity, suppose that $(G,\sigma,p)$ is universally rigid. 
Translate the configuration so that the center of gravity is at the origin.
Our idea is to introduce a slack variable for each constraint of (P). 
For this, we consider the ambient space ${\cal K}$ and convex cone ${\cal K}_+$ defined by
\begin{align*}
\mathcal{K}:={\cal L}^V\times \mathbb{R}^{E_\pm} \times \{0\}^{E_0} \text{ and }
\mathcal{K}_+:={\cal L}^V_+\times \mathbb{R}^{E_\pm}_{\geq 0} \times \{0\}^{E_0},
\end{align*}
respectively, where $E_{\pm}=E_+ \cup E_-$.
In the following discussion, an element in ${\cal K}$ is often denoted by a pair $(X,s)$ with $X \in {\cal L}^V$ and $s \in \mathbb{R}^{E_\pm} \times \{0\}^{E_0}$.
$\langle \cdot \rangle$ denotes the Euclidean inner product in ${\cal K}$.
Consider the following SDP over $\mathcal{K}$:
    \[
\begin{array}{llll}
        \text{(P')} &  \text{max.}  & 0 \\
         & \text{s.t.}  & \langle (X,s), (F_{ij}, \sigma(ij) {\bm e}_{ij}) \rangle   =  \| p_i-p_j\|^2 &(ij \in E(G)) \\
         &              & (X,s) \in \mathcal{K}_+,
    \end{array}
\]
where, for $ij \in E(G)$, ${\bm e}_{ij}$ denotes the unit vector in $\mathbb{R}^{E(G)}$ such that the $ij$-th entry is one.
Observe that $X$ is feasible in (P) if and only if  $(X,s)$ is feasible in (P') for some $s\in \mathbb{R}_{\geq 0}^{E_\pm} \times \{0\}^{E_0}$.
As $(G,\sigma,p)$ is universally rigid, Proposition~\ref{prop:uniqueness} implies that $(P^{\top}P, {\bm 0})\in \mathcal{K}$ is the unique solution of (P').

 We shall apply Proposition~\ref{prop:GT} to this setting.
 To do so, we need to prove the local genericity of $(P^{\top}P, {\bm 0})$. 
%
    \begin{claim} \label{claim:tensegrity}
        Let $k=\binom{d+1}{2}$.
        Then $(P^\top P,\bm{0})\in \mathcal{K}$ is locally generic in $\text{ext}_k(\mathcal{K}_+)$.
    \end{claim}
    \begin{proof}
        A map $f: \mathcal{C}(V) \rightarrow \mathcal{L}^V_+$; $p \mapsto P^\top P$ is algebraic over $\mathbb{Q}$ and
        $f(\mathcal{C}_d(V) )$ equals to $\mathcal{L}_{+,d}:=\{L\in \mathcal{L}^V_+: \rank L=d\}$ by (\ref{eq:conf2}).
        Hence, by Proposition \ref{prop inherit genericity}, $P^\top P$ is generic in $\bigcup_{i\leq d} \mathcal{L}_{+,i}$.
%

        From Proposition~\ref{prop:realface} for $k=\binom{d+1}{2}$, 
        \begin{equation}\label{eq:ten_proof0}
        \text{ext}_k(\mathcal{L}_+^V)=\left\{L\in \mathcal{L}^V_+:\binom{\rank L+1}{2} \leq k \right\}=\bigcup_{i\leq d} \mathcal{L}_{+,i}.
        \end{equation}
        Hence,  $P^\top P$ is generic in  $\text{ext}_k(\mathcal{L}_+^V)$.
        (\ref{eq:ten_proof0}) also implies 
        \begin{equation}\label{eq:ten_proof1}
        \text{ext}_k(\mathcal{K}_+) = \left\{ (L,s) \in \mathcal{K}_+ : \binom{\rank  L+1}{2} + \|s\|_0  \leq k \right\},
        \end{equation}
        where $\|s\|_0$ denotes the number of non-zero elements of $s$.
        By the lower semi-continuity of rank, 
        there exists a neighborhood $U$ of $P^{\top}P$ in $\mathcal{L}^V$ in which the rank of any matrix is at least $d$.
   By (\ref{eq:ten_proof1}) and $k=\binom{d+1}{2}$, we have
   \[
   \left(U \times \mathbb{R}_{\geq 0}^{E_\pm} \times \{0\}^{E_0} \right)\cap \text{ext}_k(\mathcal{K}_+) = \left\{ (L,\bm{0}) \in \mathcal{K}_+ : \rank L =d \right\}.
   \]
%
Hence $(P^\top P, {\bm 0})$ is generic in $\left(U \times \mathbb{R}_{\geq 0}^{E_\pm} \times \{0\}^{E_0} \right)\cap \text{ext}_k(\mathcal{K}_+)$, meaning that $(P^\top P, {\bm 0})$ is locally generic in $\text{ext}_k(\mathcal{K}_+)$.
    \end{proof}

    We are now in a position to complete the proof.
    We consider the subspace
    \[
    \mathcal{K}(G):=\text{span} \{ (F_{ij},\sigma(ij)\bm{e}_{ij}) :  ij\in E(G)\}
    \]
    of $\mathcal{K}$,  and let $\pi: \mathcal{K} \rightarrow \mathcal{K}(G)$ be a projection.
        Since $(P^\top P,\bm{0})$ is the unique solution of (P'), 
        $\pi^{-1}\left(\pi( (P^\top P,\bm{0}) )\right) \cap \mathcal{K}_+$ is a singleton set.
        Since $\mathcal{K}_+$ is a closed line-free convex set and 
        $(P^\top P,\bm{0})$ is locally generic in $\text{ext}_{\binom{d+1}{2}}(\mathcal{K}_+)$ by Claim~\ref{claim:tensegrity}, 
        Proposition~\ref{prop:GT} can be applied. 
        Hence,    there exists a hyperplane $H=\{(X,s)\in \mathcal{K}(G): \langle (X,s), (L,t)\rangle=0\}$ defined by $(L,t)\in \mathcal{K}(G)$ 
        such that $\pi^{-1}(H)=\{(X,s)\in \mathcal{K}: \langle (X,s), (L,t)\rangle=0\}$ exposes $F_{\mathcal{K}_+}((P^{\top}P, {\bm 0}))$.
        Note that $F_{\mathcal{K}_+}((P^{\top}P, {\bm 0}))=F_{\mathcal{L}_+}(P^{\top}P)\times \{\bm 0\}$.
        This and Proposition~\ref{prop:realface} imply  
        \begin{equation} \label{eq:ten_proof2}
        L\succeq 0,\ \rank L=n-d-1,\  \langle P^\top P, L\rangle=0,
        \end{equation}
        and  
        \begin{equation}\label{eq:ten_proof3}
        t_{ij} > 0 \qquad (ij\in E_\pm).
        \end{equation}
        
        By $(L,t)\in \mathcal{K}(G)$, $(L,t)=\sum_{ij\in E(G)} \omega_{ij} (F_{ij}, \sigma(ij){\bm e}_{ij})$ for some $\omega:E(G)\rightarrow \mathbb{R}$.
        Hence $L=L_{G,\omega}$. By $PL=O$ from (\ref{eq:ten_proof2}), $\omega$ is an equilibrium stress of $(G,\sigma,p)$.
        Moreover, (\ref{eq:ten_proof3}) implies that $\sigma(ij)\omega_{ij}=t_{ij}>0$ for every $ij\in E_\pm$, that is, $\omega$ is strictly proper.
        Together with (\ref{eq:ten_proof2}), we conclude that $\omega$ satisfies the properties of the statement.
%
%
%
%
%
%
    \end{proof}

\section{Extension to Group-Symmetric Tensegrities} \label{sec:4}
    In this section, we extend Theorem~\ref{maintheoremtens} to tensegrities with finite point group symmetry.
    We use the following notations.
    Let $\Gamma$ be a finite group with the unit element $e_\Gamma$.
    The set of $n \times m$ real matrices is denoted as $\mathbb{R}^{n \times m}$.
    For a finite set $X$ with $|X|=m$, let $\mathbb{R}^{X \times X}$ be the set of $m \times m$ real matrices whose each row or column is indexed by each element in $X$.
    For $A\in \mathbb{R}^{n\times m}$, the $(i,j)$-th entry is denoted by $A[i,j]$.
    The identity matrix of size $n$ is denoted as $I_n$.
    The general linear group and the orthogonal group of $\mathbb{R}^n$ are denoted as $GL_n(\mathbb{R})$ and $O(\mathbb{R}^n)$, respectively.
    For two matrices $A \in \mathbb{R}^{k \times l}$ and $B \in \mathbb{R}^{k' \times l'}$, their matrix direct sum $A \oplus B \in \mathbb{R}^{(k+k')\times (l+l')}$ and
    their matrix tensor product $A \otimes B \in \mathbb{R}^{kk' \times ll'}$ are defined by
    \[
        A \oplus B =
        \begin{pmatrix}
        A & O \\
        O & B 
        \end{pmatrix} 
    \text{ and }
        A \otimes B =
        \begin{bmatrix} 
            A[1,1]B & \cdots & A[1,l]B \\
            \vdots & \ddots & \vdots \\
            A[k,1]B & \cdots & A[k,l]B
        \end{bmatrix},
    \]
    respectively. For subsets of matrices $\mathcal{F}_1 \subseteq \mathbb{R}^{k \times l}$ and $\mathcal{F}_2 \subseteq \mathbb{R}^{k' \times l'}$, 
    $\mathcal{F}_1 \oplus \mathcal{F}_2 \subseteq \mathbb{R}^{(k+k') \times (l+l')}$ and 
    $I_n \otimes \mathcal{F}_1 \subseteq \mathbb{R}^{nk \times nl}$ are defined by
    \[
        \mathcal{F}_1 \oplus \mathcal{F}_2 = \left\{ A\oplus B : A \in \mathcal{F}_1, B \in \mathcal{F}_2 \right\}
    \text{ and }
        I_n \otimes \mathcal{F}_1 = \{ I_n \otimes A : A \in \mathcal{F}_1 \},
    \]
    respectively.

\subsection{Main Result}
    We define basic notions regarding group-symmetric tensegrities and then state our main theorem.

    Let $\hat{V}$ be a finite set and $\Gamma$ be a finite group.
    A {\em$\Gamma$-gain graph} on $\hat{V}$ is a directed graph $(\hat{V}, \hat{E})$ in which each edge is labeled by an element in $\Gamma$. 
An edge from a vertex $u$ to a vertex $v$ with label $\gamma\in \Gamma$ is denoted by a triple $(u,v,\gamma)$,
and we will identify $(u,v,\gamma)$ with $(v,u,\gamma^{-1})$. 

More rigorously, a $\Gamma$-gain graph is defined as a pair $(\hat{V}, \hat{E})$ of a finite set $\hat{V}$ and a subset $\hat{E}$ of $(\hat{V}\times \hat{V}\times \Gamma)/\sim$, where $\sim$ is an equivalence relation on $\hat{V}\times \hat{V}\times \Gamma$ defined by
    \[
        (u,v,\gamma) \sim (u',v',\gamma') \Longleftrightarrow (u',v',\gamma') = (u,v,\gamma) \text{ or } (u',v',\gamma')=(v,u,\gamma^{-1}).
    \]
    The {\em lift} of a $\Gamma$-gain graph $\hat{G}=(\hat{V},\hat{E})$ is an undirected graph $G=(V,E)$ on $V:=\Gamma \times \hat{V}$ such that
        $\{(\alpha,u),(\beta,v)\}$ is an edge of $E$ if and only if  $(u,v,\alpha^{-1}\beta) \in \hat{E}.$ 
For simplicity, we often denote an edge $\{(\alpha,u),(\beta,v)\}$ of the lift by $(\alpha,u)(\beta,v)$.

An undirected graph $G=(V,E)$ is said to be a {\em $\Gamma$-symmetric graph} if it is the lift of some $\Gamma$-gain graph $\hat{G}=(\hat{V},\hat{E})$. 
Figure~\ref{fig:2} is an example of a $\mathbb{Z}_2$-symmetric graph.
 A {\em Cayley graph} is a group-symmetric graph with $|\hat{V}|=1$.
A $\Gamma$-symmetric graph can be the lift of more than one $\Gamma$-gain graph $\hat{G}$. 
In the subsequent discussion, it would be convenient to pick arbitrary one $\hat{G}$ to be the {\em quotient} of $G$ and denote it by $G/\Gamma$.
For $v \in V(G/\Gamma)$, a subset of vertices $\Gamma \times \{v\}\subseteq V(G)$ is called a {\em vertex orbit}.
    For $(u,v,\gamma) \in E(G/\Gamma)$, a subset of edges $\{ (\alpha,u)(\alpha\gamma,v) : \alpha \in \Gamma\} \subseteq E$ is called an {\em edge orbit}.
    \begin{figure}
        \centering
    \begin{tikzpicture}[]
        \node[draw,circle] (1) {$v_1$};
        \node[draw,circle,below left=0.75cm and 1.05cm of 1] (2) {$v_2$};
        \node[draw,circle,below right=1.05cm and 0.75cm of 2] (3) {$v_3$};
        \draw[->] (1) -- node[midway,left]{$+1$}  (2);
        \draw[->] (1) -- node[midway,right]{$+1$} (3);
        \draw[->] (2) -- node[midway,above]{$+1$} (3);
        \draw[->] (1) edge[loop right] node[midway,right]{$-1$} (1);
        \draw[->] (2)  to[bend right] node[midway,below]{$-1$}  (3);
        \draw[->] (3) edge[loop right] node[midway,right]{$-1$} (3);
    \end{tikzpicture}
    \begin{tikzpicture}[]
        \node[draw,fill=black,circle,inner sep=1pt,minimum size=2pt] (1) {};
        \node[draw,fill=black,circle,inner sep=1pt,minimum size=2pt,below left=1.0cm and 1.4cm of 1] (2) {};
        \node[draw,fill=black,circle,inner sep=1pt,minimum size=2pt,below right=1.4cm and 1.0cm of 2] (3) {};
        \node[draw,fill=black,circle,inner sep=1pt,minimum size=2pt,right =3.0cm of 3] (4) {};
        \node[draw,fill=black,circle,inner sep=1pt,minimum size=2pt,above right=1.4cm and 1.0cm of 4] (5) {};
        \node[draw,fill=black,circle,inner sep=1pt,minimum size=2pt,above left=1.0cm and 1.4cm of 5] (6) {};
        \node[above=0.1cm of 1] (a) {$(+1,v_1)$};
        \node[left=0.1cm of 2] (b) {$(+1,v_2)$};
        \node[below=0.1cm of 3] (c) {$(+1,v_3)$};
        \node[below=0.1cm of 4] (d) {$(-1,v_3)$};
        \node[right=0.1cm of 5] (e) {$(-1,v_2)$};
        \node[above=0.1cm of 6] (f) {$(-1,v_1)$};
        \foreach \u / \v in {1/2,2/3,3/4,4/5,5/6,6/1,1/3,2/4,3/5,4/6}
        \draw (\u) -- (\v);
    \end{tikzpicture}
    \caption{An example of $\mathbb{Z}_2$-gain graph with $\mathbb{Z}_2=\{-1,+1\}$ and its lift.}
    \label{fig:2}
\end{figure}

    Next we define a group-symmetric tensegrity.
    A group homomorphism $\theta:\Gamma \rightarrow O(\mathbb{R}^d)$ is called a {\em point group}.
    A $d$-dimensional point configuration $p:V(G) \rightarrow \mathbb{R}^d$ is {\em compatible} with a point group $\theta$ if the following relations are satisfied:
    \[
        \theta(\gamma) p_{(\alpha,v)} = p_{(\gamma \alpha,v)} \qquad (\gamma \in \Gamma, (\alpha,v) \in V(G) ).
    \]
    Let $\mathcal{C}_\theta(V(G))$ (or ${\cal C}_\theta(V)$) be the set of all $d$-dimensional point configurations $p:V(G) \rightarrow \mathbb{R}^d$ such that $\sum_{i \in V(G)}p_i=\bm{0}$ and $p$ is compatible with $\theta$.
    A $d$-dimensional tensegrity $(G,\sigma,p)$ is a {\em $\theta$-symmetric tensegrity} if $(V,E_0)$, $(V,E_+)$, $(V,E_-)$ 
    are $\Gamma$-symmetric graphs and $p$ is compatible with $\theta$.

With this definition, our goal is to extend Theorem~\ref{maintheoremtens} to $\theta$-symmetric tensegrities.
Note that a priori a group-symmetric tensegrity is not generic, and Theorem~\ref{maintheoremtens} cannot be applied. 
In order to extend Theorem~\ref{maintheoremtens}, we need to introduce genericity modulo symmetry, which is commonly used in the context of infinitesimal rigidity~\cite{MT14,BW17}.
Let $\mathbb{Q}_{\theta,\Gamma}$ be the finite extension field of $\mathbb{Q}$ generated by the entries of $\theta(\gamma)$ for all $\gamma \in \Gamma$  and those of (representative)
irreducible representations of $\Gamma$ (see Subsection~\ref{sec:4.4} and Subsection~\ref{sec:5.3} for the formal definition).
A $\theta$-symmetric tensegrity $(G,\sigma,p)$ is {\em generic modulo symmetry} if the translation of $p$ is generic over $\mathbb{Q}_{\theta,\Gamma}$ in ${\cal C}_\theta(V)$.
In other words, we choose a representative vertex from each vertex orbit of $G$, and $(G,\sigma,p)$  is said to be generic modulo symmetry if the set of coordinates of the points of representative vertices is algebraically independent over $\mathbb{Q}_{\theta,\Gamma}$. 
Note that, in the latter definition, as $\mathbb{Q}_{\theta,\Gamma}$ contains the entries of $\theta$, genericity modulo symmetry is independent of the choice of representative vertices. 

    Now we are in a position to state our main result, an extension of Theorem~\ref{maintheoremtens} to group-symmetric tensegrities.
    \begin{theorem} \label{thm:sym}
        Let $(G,\sigma,p)$ be a $d$-dimensional tensegrity with $n$ vertices which is $\theta$-symmetric and is generic modulo symmetry.
        Suppose also that $p(\overline{N}_G(i))$ affinely spans $\mathbb{R}^d$ for all $i \in V(G)$.
        Then, $(G,\sigma,p)$ is universally rigid if and only if it has a strictly proper equilibrium stress $\omega$ satisfying $L_{G,\omega} \succeq 0$ and $\rank L_{G,\omega}=n-d-1$.
    \end{theorem}
    
    The sufficiency of Theorem~\ref{thm:sym} immediately follows from Theorem~\ref{thm:connelly} and Theorem~\ref{thm:alfakih}.
The necessity of Theorem~\ref{thm:sym} holds without assuming the neighbor-general position of $p$ as follows.
    \begin{theorem} \label{thm:sym_nec}
        Let $(G,\sigma,p)$ be a $d$-dimensional tensegrity with $n\geq d+2$ vertices which is $\theta$-symmetric and is generic modulo symmetry.
        Suppose also that $p$ affinely spans $\mathbb{R}^d$.
        If $(G,\sigma,p)$ is universally rigid, then it has a strictly proper equilibrium stress $\omega$ satisfying $L_{G,\omega} \succeq 0$ and $\rank L_{G,\omega}=n-d-1$.
    \end{theorem}

The subsequent discussion is devoted to the proof of Theorem~\ref{thm:sym_nec}.
In order to simplify the description, we shall focus on bar-joint frameworks $(G,p)$ rather than tensegrities $(G,\sigma,p)$. 
For general tensegrities the proof easily follows by combining the proof for the bar-joint case with that of Theorem~\ref{maintheoremtens}.
    
    A high-level idea of the proof is as follows.
    Proposition~\ref{prop:uniqueness} states that, if $(G,p)$ is universally rigid, then the Gram matrix $P^\top P$ of  $p$ is the unique feasible solution of the SDP problem (P).
    To use Gortler-Thurston's argument (Proposition~\ref{prop:GT}), the unique solution must be generic in a certain sense.
    When $(G,p)$ is group-symmetric, however, genericity in this sense does not hold in the cone of positive semidefinite Laplacian matrices $\mathcal{L}^V_+$.
    To make the unique solution generic, we restrict the ambient space to $\Gamma$-symmetric Laplacian matrices.
    To investigate the facial structure of the restricted cone, we block-diagonalize the ambient matrix space by using structure theorem (Proposition~\ref{prop:str}).
    If the original point configuration $p$ is generic modulo symmetry, the image of $P^\top P$ by this transformation is generic.
    Therefore Gortler-Thurston's argument can be applied.

In order to explain how to implement this idea, we shall first give the proof of Theorem~\ref{thm:sym_nec} for the case when every real irreducible representation of $\Gamma$ is  absolutely irreducible (see Subsection~\ref{sec:4.3} for the definition).
The proof for the general case follows the same idea but is technically more involved. We will give it in Section~\ref{sec:5}.

\subsection{Restriction of (P) to the Space of Group-Symmetric Laplacians}
In this subsection, we define the space of $\Gamma$-symmetric Laplacian matrices as a subspace of Laplacian matrices. 

Let $V$ be a finite set, and let $K_V$ be the complete graph on $V$.
If we can decompose $V$ into $V=\Gamma\times (V/\Gamma)$ for some finite set $V/\Gamma$, then 
$K_V$ is $\Gamma$-symmetric, 
where its quotient $K_V/\Gamma$ consists of the vertex set $V/\Gamma$ 
and the edge set  $((V/\Gamma) \times (V/\Gamma) \times \Gamma \setminus \{ (v,v,e_\Gamma) : v \in V/\Gamma \} )/\sim$. 
(Note that, since $K_V$ has no loop, its quotient has no loop with the identity label.)
The quotient $K_V/\Gamma$ is called the {\em complete $\Gamma$-gain graph}. 
In the subsequent discussion, each $\Gamma$-symmetric graph $G=(V,E)$ and its quotient $G/\Gamma$ are assumed to be subgraphs of $K_V$ and $K_V/\Gamma$, respectively.

    For a $\Gamma$-symmetric graph $G=(V,E)$, an edge weight $\omega:E(G) \rightarrow \mathbb{R}$ is {\em $\Gamma$-symmetric} if $\omega$ is constant on each edge orbit.
A Laplacian matrix is said to be {\em $\Gamma$-symmetric} if 
it is the Laplacian of $K_V$ weighted by  a $\Gamma$-symmetric edge weight $\omega$. 
Equivalently, $L \in {\cal L}^V$ is $\Gamma$-symmetric if and only if 
\begin{equation}\label{eq:Lsymmetric}
L[(\alpha,u),(\beta,v)]=L[(\gamma\alpha,u),(\gamma\beta,v)]
\end{equation}
    for any $(\alpha,u), (\beta,v) \in V$ and any $\gamma \in \Gamma$.
    Let $({\cal L}^V)^\Gamma$ be the set of all $\Gamma$-symmetric Laplacian matrices.
    Then $({\cal L}^V)^\Gamma$ is a linear subspace of $\mathcal{S}^V$ given by
    \[
        (\mathcal{L}^V)^\Gamma = {\rm span} \left\{ F_{(u,v,\gamma)} : (u,v,\gamma) \in E(K_V/\Gamma) \right\}
    \]
    where
    \[
        F_{(u,v,\gamma)} = \sum_{\alpha \in \Gamma} F_{(\alpha,u)(\alpha\gamma,v)}.
    \]
    Let $({\cal L}^V)^\Gamma_+$ be the set of positive semidefinite $\Gamma$-symmetric Laplacian matrices.

Let $\theta:\Gamma\rightarrow O(\mathbb{R}^d)$ be a point group. For a given $\theta$-symmetric framework $(G,p)$, we consider the following SDP problem:
    \[
        \begin{array}{llll}
        \text{(P$^\Gamma$)} &  \text{max.}  & 0 \\
                & \text{s.t.}  & \langle X, F_{(u,v,\gamma)} \rangle   =  \langle P^\top P,F_{(u,v,\gamma)} \rangle &((u,v,\gamma) \in E(G/\Gamma)) \\
                &              & X \in (\mathcal{L}^V)^\Gamma_+.
        \end{array}
    \] 
    The cone of (P$^\Gamma$) is restricted to $(\mathcal{L}^V)^\Gamma_+$. 
    We have the following.
    
    \begin{prop} \label{prop:1}
        If a framework $(G,p)$ with $p \in {\cal C}_\theta(V)$ is universally rigid and $\theta$-symmetric for some point group $\theta:\Gamma\rightarrow O(\mathbb{R}^d)$, then (P$^\Gamma$) has the unique solution $P^\top P$.
    \end{prop}
    \begin{proof}
    We first check that  the Gram matrix $P^\top P \in \mathcal{L}^V_+$ of $p$ is $\Gamma$-symmetric by checking (\ref{eq:Lsymmetric}) as follows:
    \begin{align*}
        (P^\top P)[(\gamma\alpha,u)(\gamma\beta,v)] &= {p_{(\gamma\alpha,u)}}^\top p_{(\gamma\beta,v)}  
        = {p_{(\alpha,u)}}^\top \theta(\gamma)^\top \theta(\gamma) p_{(\beta,v)} \\
        &= {p_{(\alpha,u)}}^\top p_{(\beta,v)} = (P^\top P)[(\alpha,u),(\beta,v)],
    \end{align*}
where the second equation follows from $\theta$-symmetry ans the third equation follows from $\theta(\gamma) \in O(\mathbb{R}^d)$.
    Hence $P^\top P \in (\mathcal{L}^V)^\Gamma_+$, and $P^\top P$ is feasible. 
    
To see the uniqueness, consider any $X \in ({\cal L}^V)^\Gamma$. 
Then $X$ is a linear combination of $F_{(u,v,\gamma)}$ over $(u,v,\gamma) \in E(K_V/\Gamma)$ while each $F_{(u,v,\gamma)}$ is the sum of 
$F_{(\alpha,u),(\alpha\gamma,v)}$ over $\alpha\in \Gamma$.
Hence we have \begin{equation}\label{eq:prop1}
\langle X,F_{(u,v,\gamma)} \rangle = |\Gamma| \langle X,F_{(\alpha,u)(\alpha\gamma,v)} \rangle
\end{equation} 
for any $(u,v,\gamma) \in E(G/\Gamma)$ and  $\alpha \in \Gamma$.
 
Since (\ref{eq:prop1}) holds for $P^\top P$ by  $P^\top P\in (\mathcal{L}^V)^\Gamma_+$,
$\langle X, F_{(u,v,\gamma)} \rangle   =  \langle P^\top P,F_{(u,v,\gamma)} \rangle$ for all $(u,v,\gamma) \in E(G/\Gamma)$ if and only if 
$\langle X, F_{(\alpha,u)(\alpha\gamma,v)} \rangle   =  \langle P^\top P,F_{(\alpha,u)(\alpha\gamma,v)} \rangle$ for all $(\alpha,u)(\alpha\gamma,v) \in E(G)$. 
In other words,  $X \in (\mathcal{L}^V)^\Gamma_+$ is feasible in (P$^\Gamma$)  if and only if $X$ is feasible in (P) (for $(G,p)$).
Hence the uniqueness of the feasible solution of (P$^\Gamma$) follows by applying Proposition~\ref{prop:uniqueness} to (P).
    \end{proof}    
\subsection{Block-diagonalization of (P$^\Gamma$)} \label{sec:4.3}
    To investigate the facial structure of $(\mathcal{L}^V)^\Gamma_+$, we use the structure theorem (Proposition~\ref{prop:str}) given below.

    Let $W$ be a finite dimensional real or complex vector space. Let $GL(W)$ be a general linear group of $W$.
    A {\em representation} of $\Gamma$ is a group homomorphism $\rho:\Gamma \rightarrow GL(W)$.
    $d_\rho:=\dim W$ is called the {\em degree} of $\rho$.
    Two representations $\rho:\Gamma \rightarrow GL(W)$ and $\rho':\Gamma \rightarrow GL(W')$ are {\em equivalent} if there is an isomorphism $T:W \rightarrow W'$
    satisfying $T \circ \rho(\gamma) =\rho'(\gamma) \circ T$ for all $\gamma \in \Gamma$.
    For a representation $\rho:\Gamma \rightarrow GL(W)$, a linear subspace $W'$ of $W$ is {\em $\Gamma$-invariant} if $\rho(\gamma) (W') \subseteq W'$ 
    for all $\gamma \in \Gamma$. A representation $\rho:\Gamma \rightarrow GL(W)$ is {\em irreducible} if the only $\Gamma$-invariant subspaces are $\{0\}$ and $W$.
The simplest example is the trivial representation, where $W$ is one-dimensional and every element is associated with the identity.
The trivial representation is denoted by $\tri$.

A real irreducible representation $\rho:\Gamma \rightarrow GL(W)$ is called {\em absolutely irreducible} if $\rho$ is also irreducible over $\mathbb{C}$. We also call $\Gamma$ {\em absolutely irreducible}
if every real irreducible representation of $\Gamma$ is absolutely irreducible. This is equivalent to saying that every complex irreducible representation is equivalent to a real representation. 
For example, dihedral groups and symmetric groups are absolutely irreducible while cyclic groups of order more than two are not.

In this section we shall focus on the case when $\Gamma$ is absolutely irreducible. This special case is sufficiently general to explain our technical idea, and moreover it includes finite groups that appear in applications.
Hence, throughout this section,  a real representation is simply called a representation.

    Let $\tilde{\Gamma}$ be the set of all equivalence classes of irreducible representations of $\Gamma$.
    By fixing a representative of each class, each element of $\tilde{\Gamma}$ is regarded as a representation. 
    Since every representation is equivalent to an orthogonal representation, we further assume that each $\rho \in \tilde{\Gamma}$ is an orthogonal matrix representation
    $\rho:\Gamma \rightarrow O(\mathbb{R}^{d_\rho})$.

The following is a basic relation on the degree of complex irreducible representations and it is also valid for real irreducible representations since $\Gamma$ is absolutely irreducible: 
    \begin{equation} \label{eq:d1}
        \sum_{\rho \in \tilde{\Gamma}} {d_\rho}^2 = |\Gamma|.
    \end{equation}

    We can now state the structure theorem. 
    \begin{prop} \label{prop:str}
Let $\Gamma$ be a finite group, $V/\Gamma$ be a finite set with $\hat{n}$ elements, and $V=\Gamma\times V/\Gamma$.  Suppose that $\Gamma$ is absolutely irreducible. Then there exists an orthogonal transformation $\Psi:\mathbb{R}^{V\times V} \rightarrow \mathbb{R}^{V\times V}$ satisfying
        \[
        \Psi((\mathcal{L}^V)^\Gamma)= \mathcal{L}^{V/\Gamma} \oplus \
        \bigoplus_{\rho \in \tilde{\Gamma} \setminus \{\tri\}} \left( I_{d_\rho}\otimes \mathcal{S}^{d_\rho \hat{n}} \right).
        \]
    \end{prop} 
    The proof is given in Subsection~\ref{sec:4.5}.

Following Proposition~\ref{prop:str}, for each irreducible representation $\rho \in \tilde{\Gamma}$, we define a linear space $\mathcal{K}_\rho$ and a convex cone $\mathcal{K}_{+,\rho}$ by
    \[
    \begin{array}{ll}
        \mathcal{K}_\rho := 
        \begin{cases}
            \mathcal{L}^{V/\Gamma} & (\rho = \tri) \\
            \mathcal{S}^{d_\rho \hat{n}} & (\rho \in \tilde{\Gamma} \setminus \{\tri\})
        \end{cases}
        \text{ and }
        \mathcal{K}_{+,\rho} := 
        \begin{cases}
            \mathcal{L}^{V/\Gamma}_+ & (\rho = \tri) \\
            \mathcal{S}^{d_\rho \hat{n}}_+ & (\rho \in \tilde{\Gamma} \setminus \{\tri\})
        \end{cases},
    \end{array}
    \]
    respectively. The ambient space $\mathcal{K}_\Gamma$ and the convex cone $\mathcal{K}_{+,\Gamma}$ are defined by
    \[
        \mathcal{K}_\Gamma=  \bigoplus_{\rho \in \tilde{\Gamma}} I_{d_\rho} \otimes  \mathcal{K}_\rho 
        \text{ and }
        \mathcal{K}_{+,\Gamma} =  \bigoplus_{\rho \in \tilde{\Gamma}} I_{d_\rho} \otimes  \mathcal{K}_{+,\rho},
    \]
respectively. Note that $\mathcal{K}_\Gamma= \Psi((\mathcal{L}^V)^\Gamma)$ by Proposition~\ref{prop:str}, and $\mathcal{K}_{+,\Gamma} = \Psi((\mathcal{L}^V)^\Gamma_+)$ also holds since an orthogonal transformation preserves positive semidefiniteness.

    For each $X \in {\cal K}_\Gamma$ and $\rho \in \tilde{\Gamma}$, the projected image of $X$ to ${\cal K}_\rho$ is denoted by $X_\rho$, i.e.,  
    $X= \bigoplus_{\rho \in \tilde{\Gamma}} I_{d_\rho} \otimes X_\rho$.

    Using the orthogonal transformation $\Psi$ in Proposition~\ref{prop:str}, the SDP problem (P$^\Gamma$) is transformed into the following block-diagonal form:
    \[
        \begin{array}{llll}
        \text{(P$^\Psi$)} &  \text{max.}  & 0 \\
                & \text{s.t.}  & \langle X, \Psi(F_{(u,v,\gamma)}) \rangle   =  \langle \Psi(P^\top P), \Psi(F_{(u,v,\gamma)}) \rangle &((u,v,\gamma) \in E(G/\Gamma)) \\
                &              & X \in {\cal K}_{+,\Gamma}.
        \end{array}
    \]
    Since $\Psi$ is an orthogonal transformation, Proposition~\ref{prop:1} immediately gives the following.
    \begin{prop} \label{prop:2}
        If a framework $(G,p)$ with $p \in {\cal C}_\theta(V)$ is universally rigid and $\theta$-symmetric with some point group $\theta:\Gamma\rightarrow O(\mathbb{R}^d)$, then (P$^\Psi$) has the unique solution $\Psi(P^\top P)$.
    \end{prop}

\subsection{Proof of Theorem~\ref{thm:sym_nec} for absolutely irreducible $\Gamma$} \label{sec:4.4}


Let $\theta:\Gamma\rightarrow O(\mathbb{R}^d)$ be a finite point group.
We take  an orthogonal matrix $Z_\Gamma \in O(\mathbb{R}^V)$ representing $\Psi$, i.e., $\Psi(X)=Z_\Gamma^\top X Z_\Gamma$ for any $X\in \mathbb{R}^{V\times V}$.

We define the finite extension field $\mathbb{Q}_{\theta,\Gamma}$ of $\mathbb{Q}$ generated by the entries of $\theta(\gamma)$ for all $\gamma \in \Gamma$ and the entries of $Z_\Gamma$\footnote{Although the details are ommited here, one can explicitly describes the entries of $Z_\Gamma$ in terms of those of $\rho(\gamma)$ for $\rho \in \tilde{\Gamma}$.}.
As in \cite[Section 5.5]{GT}, genericity and local genericity over $\mathbb{Q}_{\theta,\Gamma}$ are defined similarly and 
    Proposition~\ref{prop inherit genericity} and Proposition~\ref{prop:GT} can be generalized by changing the coefficient field to $\mathbb{Q}_{\theta,\Gamma}$. 

   Our strategy is to apply Proposition~\ref{prop:GT} over ${\cal K}_{+,\Gamma}$, and to do that we need to show that $\Psi(P^{\top}P)$ is generic over $\mathbb{Q}_{\theta,\Gamma}$ in the set of ``low rank'' matrices in ${\cal K}_{+,\Gamma}$. This can be done by showing that any ``low rank'' matrix in ${\cal K}_{+,\Gamma}$ is the $\Psi$-image of the Gram matrix of some $\theta$-symmetric point configuration in ${\cal C}_\theta(V)$. Since $p$ is generic in ${\cal C}_\theta(V)$, this would imply that $\Psi(P^{\top}P)$ is generic in the set of ``low rank'' matrices in ${\cal K}_{+,\Gamma}$.

A precise statement we want to show  is given as follows.
    \begin{prop} \label{prop:range}
Let $\theta:\Gamma\rightarrow O(\mathbb{R}^d)$ be a point group,
and let $m_\rho$ be the multiplicity of an irreducible representation $\rho \in \tilde{\Gamma}$ in $\theta$. 
Also, define a map $f$ by 
\[
f:\mathcal{C}_\theta(V)\ni q  \mapsto  \Psi(Q^\top Q)\in \mathcal{K}_{+,\Gamma},
\]
where $Q^\top Q$ is the Gram matrix of $q$. 
If $\Gamma$ is absolutely irreducible, then 
        \[
            f(\mathcal{C}_\theta(V)) = \left\{ X \in \mathcal{K}_{+,\Gamma} : \rank X_\rho \leq m_\rho ~ (\rho \in \tilde{\Gamma}) \right\}.
        \]
    \end{prop}
    The proof of Proposition~\ref{prop:range} is given in Subsection \ref{sec:4.5}. Using Proposition~\ref{prop:str} and Proposition~\ref{prop:range}, we can now give the proof of Theorem~\ref{thm:sym_nec} for absolutely irreducible $\Gamma$.

    \begin{proof}[Proof of Theorem~\ref{thm:sym_nec} for absolutely irreducible $\Gamma$]
Let $(G,p)$ be a $d$-dimensional $\theta$-symmetric framework with a point group $\theta:\Gamma\rightarrow O(\mathbb{R}^d)$ which is generic modulo symmetry and is universally rigid.
Suppose also that $p$ affinely spans $\mathbb{R}^d$.
Let $G/\Gamma$ be the quotient of $G$, and let $\hat{n}=|V(G/\Gamma)|$, which is the number of vertex orbits in $G$.
 
By a translation, we may suppose that the center of gravity is at the origin, i.e., $p \in {\cal C}_\theta(V)$. 
As the original configuration affinely spans $\mathbb{R}^d$, $\rank P =d$.
Let $m_\rho$ be the multiplicity of an irreducible representation $\rho \in \tilde{\Gamma}$ in $\theta$. Then,
    \begin{equation} \label{eq:a}
        \sum_{\rho \in \tilde{\Gamma}} d_\rho m_\rho=d.
    \end{equation}

    \begin{claim} \label{claim:1}
        Let $k=\sum_{\rho \in \tilde{\Gamma}} \binom{m_\rho+1}{2}$. 
        Then $\Psi(P^\top P) \in {\cal K}_\Gamma$ is locally generic over $\mathbb{Q}_{\theta,\Gamma}$ in $\text{ext}_k(\mathcal{K}_{+,\Gamma})$.
    \end{claim}
    \begin{proof}
        Since $(G,p)$ is generic modulo symmetry, $p$ is generic over $\mathbb{Q}_{\theta,\Gamma}$ in ${\cal C}_\theta(V)$.
        The map $f:q \mapsto \Psi(Q^\top Q)$ is algebraic over $\mathbb{Q}_{\theta,\Gamma}$.
        By Proposition~\ref{prop inherit genericity} over $\mathbb{Q}_{\theta,\Gamma}$, $\Psi(P^\top P)$ is generic in $f(\mathcal{C}_\theta(V))$ over $\mathbb{Q}_{\theta,\Gamma}$.
        Also, $\rank \Psi(P^\top P) = \rank P^\top P =\rank P=d$. Hence, by Proposition~\ref{prop:range} and (\ref{eq:a}), we obtain $\rank \Psi(P^\top P)_\rho=m_\rho$ for $\rho \in \tilde{\Gamma}$.

By the definition of $\mathcal{K}_{+,\Gamma}$, Proposition~\ref{prop:psdface} and Proposition~\ref{prop:realface},
        \begin{equation} \label{eq:b}
            \text{ext}_k(\mathcal{K}_{+,\Gamma}) = \left\{ X \in \mathcal{K}_{+,\Gamma} : \sum_{\rho \in \tilde{\Gamma}} \binom{\rank X_\rho +1}{2} \leq k \right\}.
        \end{equation}
        For each $\rho \in \tilde{\Gamma}$, by the lower semi-continuity of rank, there exists a neighborhood $U_\rho$ of $\Psi(P^\top P)_\rho$ in ${\cal K}_\rho$ 
        in which the rank of any matrix is at least $\rank \Psi(P^\top P)_\rho =m_\rho$.
        $U=\oplus_{\rho \in \tilde{\Gamma}} I_{d_\rho} \otimes U_\rho$ is a neighborhood of $\Psi(P^\top P)$ in ${\cal K}_\Gamma$.
        Hence, by (\ref{eq:b}) and $k=\sum_{\rho \in \tilde{\Gamma}} \binom{m_\rho+1}{2}$, we have
        \[
            U \cap \text{ext}_k(\mathcal{K}_{+,\Gamma})
            \subseteq \left\{ X \in \mathcal{K}_{+,\Gamma} : \rank X_\rho = m_\rho ~ (\rho \in \tilde{\Gamma})\right\},
        \]
which further implies, by Proposition~\ref{prop:range}, 
\[
            U \cap \text{ext}_k(\mathcal{K}_{+,\Gamma})
            \subseteq \left\{ X \in \mathcal{K}_{+,\Gamma} : \rank X_\rho = m_\rho ~ (\rho \in \tilde{\Gamma}) \right\}\subseteq f({\cal C}_{\theta}(V)).
\]
Hence, since $\Psi(P^\top P)$ is generic in $f({\cal C}_{\theta}(V))$, $\Psi(P^\top P)$ is locally generic in $\text{ext}_k(\mathcal{K}_{+,\Gamma})$.
    \end{proof}

    Consider the subspace  
    \[
        \mathcal{K}_\Gamma(G)=\text{span} \left\{\Psi(F_{(u,v,\gamma)}) : (u,v,\gamma) \in E(G/\Gamma)\right\}
    \]
    of $\mathcal{K}_\Gamma$ and the projection $\pi:\mathcal{K}_\Gamma \rightarrow \mathcal{K}_\Gamma(G)$.
    Proposition~\ref{prop:2} states that  $\Psi(P^\top P)$ is the unique solution of (P$^\Psi$),
and hence  $\pi^{-1}(\pi(\Psi(P^\top P))) \cap \mathcal{K}_{+,\Gamma}$ is a singleton.
    Since $\mathcal{K}_{+,\Gamma}$ is a closed line-free convex set and $\Psi(P^\top P)$ is locally generic in $\text{ext}_k(\mathcal{K}_{+,\Gamma})$ with $k=\sum_{\rho \in \tilde{\Gamma}} \binom{m_\rho+1}{2}$ by Claim~\ref{claim:1}, Proposition~\ref{prop:GT} can be applied.
    Hence there exists a hyperplane $H=\{\langle X,L \rangle =0 : X \in \mathcal{K}_\Gamma(G)\}$ in $\mathcal{K}_\Gamma(G)$ defined by $L \in \mathcal{K}_\Gamma(G)$
    such that $\pi^{-1}(H)=\{\langle X,L \rangle =0 : X \in \mathcal{K}_\Gamma\}$ exposes $F_{\mathcal{K}_{+,\Gamma}}(\Psi(P^\top P))$.

Following the decomposition of the ambient space ${\cal K}_{\Gamma}$, we have the decomposition of 
the face $F_{\mathcal{K}_{+,\Gamma}}(\Psi(P^\top P))$ as follows:
    \[
        F_{\mathcal{K}_{+,\Gamma}}(\Psi(P^\top P)) = \bigoplus_{\rho \in \tilde{\Gamma}} I_{d_\rho} \otimes F_{\mathcal{K}_{+,\rho}} (\Psi(P^\top P)_\rho).
    \]
In each component ${\cal K}_{\rho}$ of ${\cal K}_{\Gamma}$, $\pi^{-1}(H)$ restricted to ${\cal K}_{\rho}$ is a hyperplane of ${\cal K}_{\rho}$ given by $\{\langle X,L_\rho \rangle =0 : X \in \mathcal{K}_\Gamma\}$ and it exposes $F_{\mathcal{K}_{+,\rho}} (\Psi(P^\top P)_\rho)$.
(Recall that $L_\rho$ denotes the projected image of $L$ to ${\cal K}_\rho$.)
Since $\rank \Psi(P^\top P)_\rho=m_{\rho}$,  Proposition~\ref{prop:psdface} and Proposition~\ref{prop:realface} imply
    \begin{align}
        &\rank L_\rho = \begin{cases} \hat{n}-m_\tri-1 & (\text{if }\rho =\tri)  \\
\hat{n}d_\rho-m_\rho & (\text{if } \rho \in \tilde{\Gamma} \setminus \{\tri\}) 
\end{cases} \label{eq:symprf4} \\
        &L_\rho \succeq 0, ~ \langle \Psi(P^\top P)_\rho, L_\rho \rangle=0. \label{eq:symprf5}
    \end{align}
for every $\rho \in \tilde{\Gamma}$.
    By (\ref{eq:d1}), (\ref{eq:a}), (\ref{eq:symprf4}),
    \begin{equation} \label{eq:symprf6}
    \rank L = \hat{n}-m_\tri-1 + \sum_{\rho \in \tilde{\Gamma} \setminus \{1\}} d_\rho(\hat{n}d_\rho-m_\rho) = \hat{n}|\Gamma|-d-1 = n-d-1.
    \end{equation}
    By $L \in \mathcal{K}_\Gamma(G)$, $L=\sum_{(u,v,\gamma) \in E(G/\Gamma)} \hat{\omega}_{(u,v,\gamma)}\Psi(F_{(u,v,\gamma)})$ for some 
    $\hat{\omega}: E(G/\Gamma) \rightarrow \mathbb{R}$.
    Hence, by (\ref{eq:symprf5}) and (\ref{eq:symprf6}), $\Psi^{-1}(L) = \sum_{(u,v,\gamma) \in E(G/\Gamma)} \hat{\omega}_{(u,v,\gamma)}F_{(u,v,\gamma)}$ is a $\Gamma$-symmetric weighted Laplacian of $G$ satisfying
    \[
        \Psi^{-1}(L) \succeq 0, \quad \rank \Psi^{-1}(L)=n-d-1, \quad \langle P^\top P, \Psi^{-1}(L) \rangle =0.
    \]
    Therefore $\Psi^{-1}(L)$ satisfies the properties of the statement.
    \end{proof}
    
\subsection{Proofs of Proposition~\ref{prop:str} and Proposition~\ref{prop:range}} \label{sec:4.5}
We give a proof of Proposition~\ref{prop:str} and Proposition~\ref{prop:range}.
Let $V=\Gamma \times V/\Gamma$ be a finite set, and denote 
$n=|V|$ and $\hat{n}=|V/\Gamma|$.
In the subsequent discussion, we shall frequently look at vectors in  $\mathbb{R}^{\Gamma}\otimes \mathbb{R}^{V/\Gamma}$ (which can be identified with $\mathbb{R}^V$) and matrices in $\mathbb{R}^{\Gamma\times \Gamma}\otimes \mathbb{R}^{V/\Gamma\times V/\Gamma}$.
Let $\{\be_\gamma: \gamma\in \Gamma\}$ be the standard basis of $\mathbb{R}^{\Gamma}$ and $\{\be_v: v\in V/\Gamma\}$ be that of $\mathbb{R}^{V/\Gamma}$.
Then $\{E_{\alpha,\beta}:=\be_\alpha\be_\beta^{\top}: \alpha,\beta\in \Gamma\}$ and $\{E_{u,v}:=\be_u\be_v^\top: u,v\in V/\Gamma\}$ form a base of 
$\mathbb{R}^{\Gamma\times \Gamma}$ and a base of $\mathbb{R}^{V/\Gamma\times V/\Gamma}$, respectively.

With this notation, the {\em right regular representation} $R:\Gamma \rightarrow \mathbb{R}^{\Gamma\times \Gamma}$ of $\Gamma$ is defined by 
\begin{equation}\label{eq:regular}    
R(\gamma) = \sum_{\alpha \in \Gamma} E_{\alpha,\alpha\gamma}.
\end{equation}
A basic fact from representation theory is that, over the complex field, 
the right regular representation is (unitary) equivalent to 
$\bigoplus_{\rho} I_{d_\rho} \otimes \rho(\gamma)$,
where the sum is taken over all non-equivalent complex irreducible representations. 
Since $\Gamma$ is assumed to be absolutely irreducible, this in turn implies that there is an orthogonal matrix $Z\in O(\mathbb{R}^{\Gamma})$ such that 
\begin{equation} \label{eq:4.5-1}
            Z^\top R(\gamma) Z = \bigoplus_{\rho \in \tilde{\Gamma}} I_{d_\rho} \otimes \rho(\gamma) \quad (\gamma \in \Gamma),
\end{equation}
where recall that $\tilde{\Gamma}$ is the set of all {\em real} irreducible representations of $\Gamma$.
This orthogonal transformation gives a linear isomorphism between 
${\cal T}:={\rm span}\{R(\gamma):\gamma\in \Gamma\}$ and 
$\bigoplus_{\rho\in \tilde{\Gamma}} I_{d_{\rho}} \otimes \mathbb{R}^{d_{\rho}\times d_{\rho}}$.
Indeed, if we set 
$\psi:\mathbb{R}^{\Gamma\times \Gamma}\ni X\mapsto Z^{\top} XZ \in \mathbb{R}^{\Gamma\times \Gamma}$, then by (\ref{eq:4.5-1}),
$\psi({\cal T}) \subseteq \bigoplus_{\rho\in \tilde{\Gamma}} I_{d_{\rho}} \otimes \mathbb{R}^{d_{\rho}\times d_{\rho}}$, and 
since $Z$ is non-singular $\psi|_{\cal T}$ is an injective linear map.
By (\ref{eq:d1}), ${\cal T}$ and $\bigoplus_{\rho\in \tilde{\Gamma}} I_{d_{\rho}} \otimes \mathbb{R}^{d_{\rho}\times d_{\rho}}$ have the same dimensions,
and hence $\psi|_{\cal T}$ is an isomorphism.

Using $Z$, we consider an orthogonal map $\Psi:\mathbb{R}^{V\times V} \rightarrow \mathbb{R}^{V\times V}$ given by $\Psi(X)=(Z \otimes I_{\hat{n}})^\top X (Z \otimes I_{\hat{n}})$. We show this is the desired map for Proposition~\ref{prop:str}. 

\begin{proof}[Proof of Proposition~\ref{prop:str}]
 Since $\psi$ is a linear isomorphism between ${\cal T}$ and  $\bigoplus_{\rho\in \tilde{\Gamma}} I_{d_{\rho}} \otimes \mathbb{R}^{d_{\rho}\times d_{\rho}}$, we have
\begin{equation}\label{eq:str0}
\Psi({\cal T}\otimes \mathbb{R}^{V/\Gamma\times V/\Gamma})=\bigoplus_{\rho\in \tilde{\Gamma}} I_{d_{\rho}} \otimes \mathbb{R}^{d_{\rho}\times d_{\rho}}\otimes \mathbb{R}^{V/\Gamma\times V/\Gamma}.
\end{equation}

Also (\ref{eq:Lsymmetric}) says that a Laplacian $L\in \mathbb{R}^V$ is in $({\cal L}^V)^{\Gamma}$ if and only if it is spanned by $\{R(\gamma) \otimes E_{u,v}:\gamma\in \Gamma, u,v \in V/\Gamma\}$. In other words, 
\begin{equation}\label{eq:str01}
({\cal L}^V)^{\Gamma}={\cal L}^V\cap ({\cal T}\otimes \mathbb{R}^{V/\Gamma\times V/\Gamma}).
\end{equation}
Instead of computing $\Psi(({\cal L}^V)^{\Gamma})$, we now  compute the $\Psi$-image of  the right side of (\ref{eq:str01}). 
\begin{claim} \label{claim:str02}
$X\in {\cal T}\otimes \mathbb{R}^{V/\Gamma\times V/\Gamma}$ is Laplacian if and only if $\Psi(X)$ is symmetric and  $\Psi(X)_\tri$ is Laplacian,
where $\Psi(X)_\tri$ stands for the projected image of $\Psi(X)$ to the component associated with the trivial representation in (\ref{eq:str0}).
\end{claim}
\begin{proof}
Clearly, $X$ is symmetric if and only if $\Psi(X)$ is symmetric.
Also, a symmetric $X \in \mathcal{S}^n$ is Laplacian if and only if $X\bm{1}_n=\bm{0}$.
Hence it suffices to show that
\begin{equation}\label{eq:str03}
\text{$X\bm{1}_V=\bm{0}$ if and only if $\Psi(X)_\tri \bm{1}_{V/\Gamma}=\bm{0}$}.
\end{equation}

To see this, we use a fact that the span of the vector $\bm{1}_\Gamma\in \mathbb{R}^{\Gamma}$ is 
the subrepsentation of $R$ and corresponds to the trivial representation. (This can be checked by $R(\gamma)\bm{1}_\Gamma=\bm{1}_\Gamma$ for any $\gamma\in \Gamma$.)
Hence $Z^\top \bm{1}_\Gamma =|\Gamma|\bm{e}_\tri$, where $\bm{e}_\tri$ is a unit vector in $\mathbb{R}^\Gamma$.
This implies that 
\begin{align*}
(Z^{\top}\otimes I_{V/\Gamma}) \bm{1}_V
=(Z^{\top}\otimes I_{V/\Gamma}) (\bm{1}_\Gamma \otimes \bm{1}_{V/\Gamma})
=(Z^{\top}\bm{1}_{\Gamma})\otimes \bm{1}_{V/\Gamma} 
=|\Gamma| \bm{e}_{\tri}\otimes \bm{1}_{V/\Gamma}.
\end{align*}
Since $X\bm{1}_V=\bm{0}$ if and only if $\Psi(X)(Z^{\top}\otimes I_{V/\Gamma})\bm{1}_V=\bm{0}$, 
we obtain (\ref{eq:str03}), and hence the claim follows.
%
\end{proof}

Combining the relations obtained so far, we get
\begin{align*}
\Psi(({\cal L}^V)^{\Gamma})&=\Psi({\cal L}^V\cap ({\cal T}\otimes \mathbb{R}^{V/\Gamma\times V/\Gamma})) & (\text{by (\ref{eq:str01}}))\\
&=\left\{X: X\in \bigoplus_{\rho\in \tilde{\Gamma}} 
I_{d_{\rho}}\otimes \mathbb{R}^{d_\rho\times d_{\rho}}\otimes \mathbb{R}^{V/\Gamma\times V/\Gamma}, X \in {\cal S}^V, X_\tri \bm{1}_{V/\Gamma}=\bm{0} \right\} & (\text{by (\ref{eq:str0}) and Claim~\ref{claim:str02}}) 
\\
&=\left\{X: X\in \bigoplus_{\rho\in \tilde{\Gamma}}I_{d_{\rho}}\otimes \mathcal{S}^{d_\rho\hat{n}}, X_\tri \bm{1}_{V/\Gamma}=\bm{0} \right\} \\
&=\left\{X: X\in\left( \mathcal{L}^{V/\Gamma}\oplus \bigoplus_{\rho\in \tilde{\Gamma} \setminus \{\tri\}}I_{d_{\rho}}\otimes \mathcal{S}^{d_\rho\hat{n}}  \right)\right\}
\end{align*}
as required.
\end{proof}

    Next we prove Proposition~\ref{prop:range}. In the proof, we use the following orthogonality relation of complex irreducible representations (see, e.g.,~\cite{serre1977linear} for details).
    \begin{prop}[Schur orthogonality] \label{prop:schur}
        Let $\pi:\Gamma \rightarrow GL_{d_\pi}(\mathbb{C})$ and $\pi':\Gamma \rightarrow GL_{d_{\pi'}}(\mathbb{C})$ be complex irreducible representations.
        If $\pi$ and $\pi'$ are not equivalent, then 
\[
\sum_{\gamma \in \Gamma} (\pi(\gamma)[k,l]) (\overline{\pi'(\gamma)[k',l']}) =0\] 
for any $1 \leq k,l \leq d_\pi$ and $1 \leq k',l' \leq d_{\pi'}$.
        If $\pi=\pi'$, then
        \[
            \sum_{\gamma \in \Gamma} (\pi(\gamma)[k,l]) (\overline{\pi(\gamma)[k',l']}) = 
            \begin{cases} 
                \frac{|\Gamma|}{d_\pi} & ( k=k', l=l'), \\
                0 & (\text{otherwise}) 
            \end{cases}
        \]
 for any $1 \leq k,l,k',l' \leq d_\pi$.
        \end{prop}
    Let $\Gamma$ be an absolutely irreducible  finite group  and $(G,p)$ be a $d$-dimensional $\theta$-symmetric framework with $\theta:\Gamma \rightarrow O(\mathbb{R}^d)$.
We fix a representative vertex from each vertex orbit, 
and let $\tilde{P}$ be a $d \times \hat{n}$ matrix given by arranging the coordinates of the representative vertices.
Then we may suppose (by permuting the columns appropriately)
\begin{equation}
\label{eq:4.4.1}
P=\sum_{\gamma \in \Gamma} \bm{e}_\gamma^\top \otimes (\theta(\gamma) \tilde{P}).
\end{equation}

Since $\theta$ is an orthogonal representation, there is an orthogonal matrix $Y \in O(\mathbb{R}^d)$ such that 
    \begin{equation} \label{eq:f}
        Y \theta(\gamma) Y^\top = \bigoplus_{\rho \in \tilde{\Gamma}} I_{m_\rho} \otimes \rho(\gamma),
\end{equation}
where recall that $m_\rho$ denotes the multiplicity of $\rho \in \tilde{\Gamma}$ in $\theta$.
   
 Let $S$ be a set of indices defined by
    \[
        S = \left\{ (\rho,t,k) : \rho \in \tilde{\Gamma}, 1 \leq t \leq m_\rho, 1 \leq k \leq d_\rho \right\}.
    \]
By (\ref{eq:a}), $|S|=d$. Hence in the following discussion we identify $\mathbb{R}^S$ with $\mathbb{R}^d$ 
by taking the standard basis $\{\be_{(\rho,t,k)}:(\rho,t,k)\in S\}$. 

With these notations, we can calculate $\Psi(P^\top P)$ explicitly.
    \begin{lemma} \label{lem:1}
For each $\rho\in \tilde{\Gamma}$ and $t$ with $1\leq t\leq m_{\rho}$, define $w_{\rho,t} \in \mathbb{R}^{d_\rho}\otimes \mathbb{R}^{V/\Gamma}$ by
        \begin{equation}\label{eq:4.7}
            w_{\rho,t}= \sqrt{\frac{|\Gamma|}{d_\rho}} \sum_{1 \leq k \leq d_\rho} \bm{e}_k \otimes \left(\tilde{P}^\top Y^\top \bm{e}_{(\rho,t,k)} \right).
        \end{equation}
Then, for each $\rho\in \tilde{\Gamma}$, 
        \[
        \Psi(P^\top P)_\rho = \sum_{1 \leq t \leq m_\rho} w_{\rho,t} {w_{\rho,t}}^\top.
        \]
    \end{lemma}

    \begin{proof}
        By (\ref{eq:regular}) and (\ref{eq:4.4.1}), we have
        \begin{equation} \label{eq:lem1-1}
            P^\top P
            =\sum_{\alpha,\beta \in \Gamma} E_{\alpha,\beta} \otimes ( \tilde{P}^\top \theta(\alpha^{-1} \beta) \tilde{P})
            =\sum_{\gamma \in \Gamma} R(\gamma) \otimes \left( \tilde{P}^\top \theta(\gamma) \tilde{P} \right).
        \end{equation}
Hence
        \begin{align*} 
            \Psi(P^\top P)_\rho &= \sum_{\gamma \in \Gamma} \rho(\gamma) \otimes \left( \tilde{P}^\top \theta(\gamma) \tilde{P} \right) & 
(\text{by (\ref{eq:4.5-1}) and (\ref{eq:lem1-1})})\\
            &=\sum_{\gamma \in \Gamma} (I_{d_\rho}\otimes \tilde{P})^\top \left( \rho(\gamma) \otimes \theta(\gamma) \right) (I_{d_\rho} \otimes \tilde{P}) \\
            &=(I_{d_\rho} \otimes \tilde{P}^\top Y^\top) \left( \sum_{\gamma \in \Gamma} \rho(\gamma) \otimes \bigoplus_{\rho' \in \tilde{\Gamma}}
            \bigoplus_{1 \leq t \leq m_{\rho'}} \rho'(\gamma) \right) ( I_{d_\rho} \otimes Y\tilde{P}) & (\text{by (\ref{eq:f})}). \label{eq:d}
        \end{align*}
Note that, by Schur orthogonality (Proposition~\ref{prop:schur}) and the absolute irreducibility of $\Gamma$, for any  $\rho, \rho' \in \tilde{\Gamma}$,
        \[
            \sum_{\gamma \in \Gamma} \rho(\gamma) \otimes \rho'(\gamma) =
            \begin{cases}
                \frac{|\Gamma|}{d_\rho} \sum_{1 \leq k,l \leq d_\rho} E_{kl} \otimes E_{kl} & (\rho = \rho'), \\
                O & (\rho \neq \rho').
            \end{cases}
        \]
        Hence, we have
        \begin{align*} 
            \sum_{\gamma \in \Gamma} \rho(\gamma) \otimes \bigoplus_{\rho' \in \tilde{\Gamma}} \bigoplus_{1 \leq t \leq m_{\rho'}} \rho'(\gamma)
            &= \frac{|\Gamma|}{d_\rho} \sum_{1\leq t \leq m_{\rho}} \sum_{1 \leq k,l \leq d_\rho} E_{kl} \otimes E_{(\rho,t,k)(\rho,t,l)} \\
&= \frac{|\Gamma|}{d_\rho} \sum_{1\leq t \leq m_{\rho}} \left(\sum_{1 \leq k \leq d_\rho} \bm{e}_k \otimes \bm{e}_{(\rho,t,k)}\right) 
            \left(\sum_{1 \leq l \leq d_\rho} \bm{e}_l \otimes \bm{e}_{(\rho,t,l)}\right)^\top.
        \end{align*}
 Therefore, we  obtain $\Psi(P^{\top} P)=\sum_{1\leq t\leq m_{\rho}} w_{\rho,t} w_{\rho,t}^\top$ with
\[
w_{\rho,t}=(I_{d_\rho} \otimes \tilde{P}^\top Y^\top) \left(\sqrt{\frac{|\Gamma|}{d_\rho}}\sum_{1 \leq k \leq d_\rho} \bm{e}_k \otimes \bm{e}_{(\rho,t,k)}\right)
=\sqrt{\frac{|\Gamma|}{d_\rho}} \sum_{1 \leq k \leq d_\rho} \bm{e}_k \otimes \left(\tilde{P}^\top Y^\top \bm{e}_{(\rho,t,k)}\right)
\]
 as required.
    \end{proof}

    \begin{proof}[Proof of Proposition \ref{prop:range}]
        For any $p \in \mathcal{C}_\theta(V)$, by Lemma~\ref{lem:1}, $\rank \Psi(P^\top P)_\rho \leq m_\rho$.
        So, it is sufficient to prove that for any $X \in \mathcal{K}_{+,\Gamma}$ satisfying $\rank X_\rho \leq m_\rho$ ($\rho \in \tilde{\Gamma}$),
        there exists $q \in \mathcal{C}_\theta(V)$ satisfying $\Psi(Q^\top Q)=X$.

        Since $\rank X_\rho \leq m_\rho$ and $X_\rho$ is positive semidefinite, there exist $v_{\rho,1}, \ldots, v_{\rho,m_\rho} \in \mathbb{R}^{d_\rho \hat{n}}$ such that
        $X_\rho = \sum_{1 \leq t \leq m_\rho} v_{\rho,t} {v_{\rho,t}}^\top$.
\begin{claim}\label{claim:4.4}
There exists a matrix $\tilde{Q}\in \mathbb{R}^{d\times \hat{n}}$ such that
\[
v_{\rho,t}=\sqrt{\frac{|\Gamma|}{d_{\rho}}}
\sum_{1\leq k\leq d_{\rho}} 
\be_k \otimes \left( \tilde{Q}^{\top} Y^{\top} \be_{(\rho,t,k)} \right)
\]
for all $\rho\in \tilde{\Gamma}$ and $t$ with $1\leq t\leq m_{\rho}$.
\end{claim}
\begin{proof}
We decompose $\mathbb{R}^{d_{\rho}\hat{n}}$ into $d_{\rho}$ copies of $\mathbb{R}^{\hat{n}}$ such that 
$\mathbb{R}^{d_{\rho}\hat{n}}=\mathbb{R}^{\hat{n}}\times \cdots \times \mathbb{R}^{\hat{n}}$.
Then each $v_{\rho,t}\in \mathbb{R}^{d_{\rho}\hat{n}}$ is written by 
$v_{\rho,t}=(v_{\rho,t,1},v_{\rho,t,2},\dots, v_{\rho,t,d_{\rho}})$
for some $v_{\rho,t,1}, \dots, v_{\rho,t,d_{\rho}}\in \mathbb{R}^{\hat{n}}$.
Let $U$ be the matrix obtained by aligning 
$v_{\rho,t,k}$ for all  $(\rho,t,k)\in S$.
Since $|S|=d$, the size of $U$ is $\hat{n}\times d$.
Moreover, $\sum_{1\leq k\leq d_\rho} \be_k\otimes (U\be_{(\rho,t,k)})=v_{\rho,t}$.
Hence, by setting $\tilde{Q}=\sqrt{\frac{d_{\rho}}{|\Gamma|}}Y^{-1} U^{\top}$, we obtain the claimed $\tilde{Q}$.
\end{proof}

Let $\tilde{Q}$ be as shown in Claim~\ref{claim:4.4}, and $Q=\sum_{\gamma \in \Gamma}\bm{e}_\gamma^\top \otimes (\theta(\gamma)\tilde{Q})$.
Ths size of $Q$ is $d \times n$ and it can be identified with a configuration $q:V \rightarrow \mathbb{R}^d$.
by $Q=\sum_{\gamma \in \Gamma}\bm{e}_\gamma^\top \otimes (\theta(\gamma)\tilde{Q})$.
Then $q$ is compatible with $\theta$ and satisfies $\Psi(Q^\top Q)=X$ 
by Lemma~\ref{lem:1} and Claim~\ref{claim:4.4}. 
Since $Q^{\top}Q=\Psi^{-1}(X)$, we have $Q^{\top}Q\in {\cal L}^V$, implying  $Q\bm{1}_V=\bm{0}$. Hence $q \in {\cal C}_\theta(V)$. This completes the proof.
\end{proof}

\section{Complete Proof of Theorem~\ref{thm:sym}} \label{sec:5}

In the last section we proved Theorem~\ref{thm:sym_nec} for absolutely irreducible $\Gamma$. 
In this section, we give a complete proof of Theorem~\ref{thm:sym_nec}.
    If $\Gamma$ is not absolutely irreducible, the corresponding structure theorem has a slightly more involved form (Proposition~\ref{prop:str2}),
and accordingly we need a further technical analysis of the facial structure of the cone of positive semidefinite $\Gamma$-symmetric Laplacian matrices.

Throughout this section, we distinguish real representations and complex representations, and denote the equivalence classes of real irreducible representations of $\Gamma$ by $\tilde{\Gamma}$.
Again, we may assume that each $\rho \in \tilde{\Gamma}$ is an orthogonal matrix representation $\rho :\Gamma \rightarrow O(\mathbb{R}^{d_\rho})$.
    Let $\mathbb{H}$ be the algebra of quaternions.
    For $x=a+b\mi+c\mj+d\mk \in \mathbb{H}$ ($a,b,c,d \in \mathbb{R}$), its conjugate $x^*$ is  $x^*=a-b\mi-c\mj-d\mk$.
\subsection{Block-diagonalization} \label{sec:5.1}

We first review a basic fact on real representation. 
\cite{B09,murota2010numerical} give more detailed algebraic expositions for applications to semidefinite programming problems.

Let $\rho:\Gamma \rightarrow GL(W)$ be a real irreducible representation. The set of isomorphisms of $W$ commutative to $\rho(\gamma)$ for all $\gamma \in \Gamma$ is denoted by $\Hom(\rho,\rho)$.
$\Hom(\rho,\rho)$ forms a division algebra over $\mathbb{R}$, 
called the {\em commutative algebra}, 
and it is known (by Frobenius theorem) that $\Hom(\rho,\rho)$ is isomorphic to either one of $\mathbb{R}, \mathbb{C}, \mathbb{H}$.
    Accordingly $\rho$ is said to be of {\em real}, {\em complex}, {\em quaternionic} type, respectively. 
For a finite group $\Gamma$, 
the set of real irreducible representations of real (resp., complex, quaternionic) type is denoted by $\tilde{\Gamma}_\mathbb{R}$ (resp., $\tilde{\Gamma}_\mathbb{C}$, $\tilde{\Gamma}_\mathbb{H}$).
Note that $\tilde{\Gamma}=\tilde{\Gamma}_\mathbb{R}\cup \tilde{\Gamma}_\mathbb{C}\cup \tilde{\Gamma}_\mathbb{H}$.

The types of real representations can be also understood in terms of decomposability over $\mathbb{C}$ as follows.
\begin{prop}[{See, e.g.,~\cite[Chapter 13.2]{serre1977linear}}]\label{prop:type}
For a real irreducible representation $\rho\in \tilde{\Gamma}$, 
\begin{itemize}
\item $\rho$ is of real type if and only if it is irreducible over $\mathbb{C}$ (i.e. absolutely irreducible);
\item $\rho$ is of complex type if and only if it is decomposed over $\mathbb{C}$ into the direct sum of a complex irreducible representation
    $\pi$ and its complex conjugate $\overline{\pi}$ such that 
    $\pi$ and $\overline{\pi}$ are not equivalent to each other;
\item $\rho$ is of quaternionic type if and only if it is decomposed over $\mathbb{C}$ into the direct sum of two copies of a self-conjugate complex irreducible representation $\pi$, i.e., $\pi=\overline{\pi}$. 
\end{itemize}
\end{prop}

A trivial representation $\tri$ is always of real type, and hence 
$\tilde{\Gamma}_\mathbb{R}\neq \emptyset$. 
An absolutely irreducible group we have studied in the last section is the case when $\tilde{\Gamma}_\mathbb{C}=\tilde{\Gamma}_\mathbb{H}=\emptyset$.
A simplest example of a non-absolutely irreducible group is the cyclic  group $C_n$ of order $n$ with $n \geq 3$, which has real irreducible representation of complex type of degree $2$.
 Another fundamental example is the quaternion group $Q_8$ whose real representations consist of four real type representations of degree $1$ and one quaternionic type representations of degree $4$.

    To describe the structure theorem, we introduce two linear matrix spaces. 
For a complex number $a+b\mi\in \mathbb{C}$ with $a,b\in \mathbb{R}$, let 
   \[
        C(a+b\mi) = \begin{pmatrix} a& -b \\ b & a \end{pmatrix}.
    \]
The map $C$ can be extended to $C:\mathbb{C}^{n\times n}\rightarrow \mathbb{R}^{2n\times 2n}$ by applying it entry-wise,
i.e., 
\[
C(Z)=\begin{pmatrix}
            C(z_{11}) & \cdots & C(z_{1n}) \\
            \vdots & \ddots & \vdots \\
            C(z_{n1}) & \cdots & C(z_{nn})
        \end{pmatrix}
\]
for $Z=(z_{ij})\in \mathbb{C}^{n\times n}$.
The space of all real expressions of $n \times n$ complex matrices is denoted by $\mathcal{C}^{2n}$, i.e., $\mathcal{C}^{2n}=C(\mathbb{C}^{n\times n})$.
Similarly, for $a+b\mi+c\mj+d\mk \in \mathbb{H}$ with $a, b, c, d \in \mathbb{R}$, let
\[
        H(a+b\mi+c\mj+d\mk) = \begin{pmatrix} a & -b & c & -d \\ b & a & d & c \\ -c & -d & a & b \\ d & -c & -b & a \end{pmatrix},
    \]
 and extend it over $\mathbb{H}^{n\times n}$ by 
\[
H(X)=\begin{pmatrix}
            H(x_{11}) & \cdots &H(x_{1n}) \\
            \vdots & \ddots & \vdots \\
            H(x_{n1}) & \cdots & H(x_{nn})
        \end{pmatrix}
\]
for $X=(x_{ij})\in \mathbb{H}^{n\times n}$.
Let ${\cal H}^{4n}=H(\mathbb{H}^{n\times n})$.

  Note that $C$ and $H$ defined above are commutative to matrix multiplication, and $C(X^*)=C(X)^\top$ and $H(Y^*)=H(Y) ^\top$ hold, where $^*$ denotes conjugate transpose.
    For $A \in \mathcal{C}^{2n}$ (resp., $A \in {\cal H}^{4n}$), the matrix $X \in \mathbb{C}^{n \times n}$  satisfying $C(X)=A$ (reps., $X \in \mathbb{H}^{n \times n}$ satisfying $H(X)=A$) is denoted by $C^{-1}(A)$ (resp., $H^{-1}(A)$).

    For each $\rho \in \tilde{\Gamma}$, define a linear space $\mathcal{K}_\rho$ by 
    \[
        \mathcal{K}_\rho = 
        \begin{cases}
            \mathcal{L}^{V/\Gamma} & (\rho=\tri) \\
            \mathcal{S}^{d_\rho \hat{n}} & (\rho \in \tilde{\Gamma}_\mathbb{R} \setminus \{\tri\}) \\
            \mathcal{C}^{d_\rho \hat{n}} \cap \mathcal{S}^{d_\rho \hat{n}} & (\rho \in \tilde{\Gamma}_\mathbb{C})  \\
            \mathcal{H}^{d_\rho \hat{n}} \cap \mathcal{S}^{d_\rho \hat{n}} & (\rho \in \tilde{\Gamma}_\mathbb{H})
        \end{cases},
    \]
and let 
    \begin{equation}\label{eq:K_G}
        \mathcal{K}_\Gamma=\bigoplus_{\mathbb{F}\in\{ \mathbb{R}, \mathbb{C}, \mathbb{H}\}} \bigoplus_{\rho \in \tilde{\Gamma}_\mathbb{F}}
        I_{ \frac{d_\rho}{\dim \mathbb{F}}} \otimes \mathcal{K}_\rho,
    \end{equation}
where $\dim \mathbb{R}, \dim \mathbb{C}, \dim \mathbb{H} = 1, 2, 4$, respectively. Then we have the following structure theorem.
Recall that $\{E_{u,v}: u,v\in V/\Gamma\}$ stands for the standard basis of $\mathbb{R}^{V/\Gamma\times V/\Gamma}$.   

    \begin{prop}\label{prop:str2}
        There exists an orthogonal transformation 
$\Psi:\mathbb{R}^{V\times V} \rightarrow \mathbb{R}^{V\times V}$ satisfying $\Psi((\mathcal{L}^V)^\Gamma) =\mathcal{K}_\Gamma$ and 
        \[
            \Psi(R(\gamma)\otimes E_{u,v}) = \bigoplus_{\mathbb{F}\in  \{\mathbb{R}, \mathbb{C}, \mathbb{H}\}} \bigoplus_{\rho \in \tilde{\Gamma}_\mathbb{F}} 
            I_{ \frac{d_\rho}{\dim \mathbb{F}}} \otimes \rho(\gamma) \otimes E_{u,v}.
        \]
    \end{prop}
    In the proof of Proposition~\ref{prop:str2}, we use the fact that real irreducible representations of complex and quaternionic type can be represented by matrices in ${\cal C}^{d_{\rho}}$ and ${\cal H}^{d_\rho}$.
    This fact follows from the following lemma.
    \begin{lemma} \label{lem:good}
        \begin{itemize}
            \item[(i)] For each real irreducible representation $\rho:\Gamma \rightarrow GL(V)$ of complex type, 
            there exists a basis $B$ of $V$ such that the matrix expression of $\rho(\gamma)$ with respect to $B$ satisfies $\rho(\gamma) \in \mathcal{C}^{d_\rho}$.
            \item[(ii)] For each real irreducible representation $\rho:\Gamma \rightarrow GL(V)$ of quaternionic type,
            there exists a basis $B$ of $V$ such that the matrix expression of $\rho(\gamma)$ with respect to $B$ satisfies $\rho(\gamma) \in \mathcal{H}^{d_\rho}$.
        \end{itemize}
    \end{lemma}
    \begin{proof}

We prove (i). As the commutative algebra $\Hom (\rho,\rho)$ of $\rho$ is isomorphic to $\mathbb{C}$, there exists $J \in \Hom (\rho, \rho)$ satisfying $J^2=-\id_V$.
            Let $B=\emptyset$. 
            We pick any non-zero $v_1$ from $V$ and add $v_1$ and $J(v_1)$ to $B$. We then pick $v_2$ from the complement of ${\rm span}B$ and add $v_2$ and $J(v_2)$ to $B$. We keep the procedure until we get  
$B=\{v_1, J(v_1), v_2, J(v_2), \ldots, v_{d_{\rho}/2}, J( v_{d_{\rho}/2})\}$. By $J^2=-\id_V$, one can easily check that $B=\{v_1, J(v_1), v_2, J(v_2), \ldots \}$ is a basis of $V$.
            As $J$ is commutative to $\rho(\gamma)$, we can deduce that matrix expression of $\rho(\gamma)$ with respect to $B$ have the desired form.

The proof for (ii) is identical.
%
    \end{proof}

In view of Lemma~\ref{lem:good}, we may suppose that 
each $\rho \in \tilde{\Gamma}_\mathbb{C}$ (resp., $\rho\in \tilde{\Gamma}_\mathbb{H}$) is an orthogonal matrix representation with $\rho \in {\cal C}^{d_{\rho}}$ (resp., $\rho \in {\cal H}^{d_{\rho}}$).
%

We now consider the decomposition of the regular representation $R$ into the real irreducible representations. 
        Recall first that the multiplicity of a complex irreducible representation $\pi$ in $R$ is equal to the degree of $\pi$.
A representation $\rho \in \tilde{\Gamma}_\mathbb{R}$ of real type is absolutely irreducible, and hence its multiplicity in $R$ is its degree $d_\rho$.
A representation $\rho \in \tilde{\Gamma}_\mathbb{C}$ of complex type  the direcist sum of a complex irreducible representation $\pi$ and its conjugate $\overline{\pi}$. 
Hence the degree of $\pi$ is $\frac{d_\rho}{2}$, 
and the multiplicity of $\rho$ in $R$ is $\frac{d_\rho}{2}$.
A representation $\rho \in \tilde{\Gamma}_\mathbb{H}$ of quaternionic type decomposes into two copies of a self-conjugate complex irreducible representation $\pi$.
Hence the degree of $\pi$ is $\frac{d_\rho}{2}$,
and the multiplicity of $\rho$ in $R$ is $\frac{d_\rho}{4}$.

Therefore, there exists an orthogonal matrix $Z \in O(\mathbb{R}^\Gamma)$ satisfying 
        \begin{equation} \label{eq:5-0}
            Z^\top R(\gamma) Z = \bigoplus_{\mathbb{F}\in\{ \mathbb{R}, \mathbb{C}, \mathbb{H}\}} \bigoplus_{\rho \in \tilde{\Gamma}_\mathbb{F}} 
            I_{ \frac{d_\rho}{\dim \mathbb{F}}} \otimes \rho(\gamma).
        \end{equation}
 Note also that by comparing the degree of $R$ and that of its decomposition into real irreducible representations, we have
        \begin{equation} \label{eq:5-1}
            \sum_{\mathbb{F} \in\{ \mathbb{R}, \mathbb{C}, \mathbb{H}\}}
            \sum_{\rho \in \tilde{\Gamma}_\mathbb{F}} \frac{{d_\rho}^2}{\dim \mathbb{F}} = |\Gamma|.
        \end{equation}       

        For each $\rho \in \tilde{\Gamma}$, let ${\cal K}'_\rho$ be a linear space defined by
        \[
            {\cal K}'_\rho =
            \begin{cases}
                \mathbb{R}^{d_\rho\times d_{\rho}} & (\rho \in \tilde{\Gamma}_\mathbb{R}) \\
                {\cal C}^{d_\rho} & (\rho \in \tilde{\Gamma}_\mathbb{C})\\
                {\cal H}^{d_\rho} & (\rho \in \tilde{\Gamma}_\mathbb{H})
            \end{cases}.
        \]
(\ref{eq:5-0}) implies that, by an orthogonal transformation, $R(\gamma)$ is mapped to $\bigoplus_{\mathbb{F}\in \{\mathbb{R},\mathbb{C},\mathbb{H}\}}\bigoplus_{\rho\in \tilde{\Gamma}_{\mathbb{F}}} I_{\frac{d_{\rho}}{\dim \mathbb{F}}}\otimes {\cal K}_{\rho}'$.
If we set ${\cal T} = {\rm span}\{R(\gamma):\gamma \in \Gamma\}$,
then a map $\psi:{\cal T} \rightarrow \bigoplus_{\mathbb{F} \in \{ \mathbb{R}, \mathbb{C}, \mathbb{H} \}} \bigoplus_{\rho \in \tilde{\Gamma}_\mathbb{F}} 
        I_{ \frac{d_\rho}{\dim \mathbb{F}}} \otimes {\cal K}'_\rho$ defined by  $\psi(X)= Z^\top X Z$ for $X\in \mathbb{R}^{\Gamma\times \Gamma}$ is a linear isomorphism between ${\cal T}$ and $\bigoplus_{\mathbb{F}\in \{\mathbb{R},\mathbb{C},\mathbb{H}\}}\bigoplus_{\rho\in \tilde{\Gamma}_{\mathbb{F}}} I_{\frac{d_{\rho}}{\dim \mathbb{F}}}\otimes {\cal K}_{\rho}'$.

Based on $\psi$, we construct an orthogonal transformation $\Psi: \mathbb{R}^{V\times V} \rightarrow \mathbb{R}^{V\times V}$ by 
$X \mapsto (Z \otimes I_{\hat{n}})X(Z \otimes I_{\hat{n}})$ for 
$X \in \mathbb{R}^{V\times V}=\mathbb{R}^{\Gamma\times \Gamma}\otimes \mathbb{R}^{V/\Gamma\times V/\Gamma}$.
Then the restriction of $\Psi$ to $({\cal L}^V)^{\Gamma}$ would be a map claimed in Proposition~\ref{prop:str2}.
Since the proof is identical to that of Proposition~\ref{prop:str}, we omit the detailed description.
\subsection{Extending Proposition~\ref{prop:range}}
    In this subsection, we prove Proposition~\ref{prop:range2}, which extends Proposition~\ref{prop:range} to general finite groups.
Recall that the projection of $X \in \mathcal{K}_\Gamma$ to $\mathcal{K}_\rho$ is denoted by $X_\rho$.
Given a point group $\theta:\Gamma \rightarrow O(\mathbb{R}^d)$, the multiplicity  of $\rho \in \tilde{\Gamma}$ in $\theta$ is denoted by $m_\rho$.
    We have
    \begin{equation} \label{eq:5-2}
        d = \sum_{\rho \in \tilde{\Gamma}}d_\rho m_\rho.
    \end{equation}
Let ${\cal K}_{+,\Gamma}=\Psi((\mathcal{L}^V)_+^\Gamma)$. 
The following proposition extends Proposition~\ref{prop:range} to general groups.
    \begin{prop} \label{prop:range2}
        For a map $f:\mathcal{C}_\theta(V) \rightarrow \mathcal{K}_{+,\Gamma};q\mapsto \Psi(Q^\top Q)$,
        \[
            f(\mathcal{C}_\theta(V)) = \left\{ X \in \mathcal{K}_{+,\Gamma} 
            : \rank X_\rho \leq \dim \mathbb{F} \cdot m_\rho ~ (\rho \in \tilde{\Gamma}_\mathbb{F}, \mathbb{F}\in \{\mathbb{R}, \mathbb{C}, \mathbb{H}\}) \right\}.
        \]
    \end{prop}
    
    In the proof of Proposition~\ref{prop:range2}, we use the explicit form of $\Psi(P^\top P)$ (Lemma~\ref{lem:cal2}).
For this, as in Section~\ref{sec:4.5}, we define $\tilde{P}$ and $Y$ as follows.
We fix a representative vertex from each vertex orbit, 
and let $\tilde{P}$ be a $d \times \hat{n}$ matrix given by arranging the coordinates of the representative vertices.
Then we may suppose (by permuting the columns appropriately)
\begin{equation}
\label{eq:5.3.1}
P=\sum_{\gamma \in \Gamma} \bm{e}_\gamma^\top \otimes (\theta(\gamma) \tilde{P}).
\end{equation}
Since $\theta$ is an orthogonal representation, there is an orthogonal matrix $Y \in O(\mathbb{R}^d)$ such that 
    \begin{equation} \label{eq:Y2}
        Y \theta(\gamma) Y^\top = \bigoplus_{\rho \in \tilde{\Gamma}} I_{m_\rho} \otimes \rho(\gamma).
\end{equation}

    Let $S$ be the set of indices defined by 
\begin{align*}
        S=
\bigcup_{\mathbb{F}\in \{\mathbb{R},\mathbb{C},\mathbb{H}\}} \left\{ (\rho,t,l,a) : \rho \in \tilde{\Gamma}_\mathbb{F}, 1 \leq t \leq m_\rho, 
        1 \leq l \leq \frac{d_\rho}{\dim \mathbb{F}}, 1\leq a \leq \dim\mathbb{F} \right\}.
    \end{align*}
By (\ref{eq:5-2}), $|S|=d$. Hence in the following discussion we identify  $\mathbb{R}^S$ with $\mathbb{R}^d$ by taking the standard basis $\{\be_{(\rho,t,k,a)}: (\rho,t,k,a)\in S\}$.
\begin{lemma}\label{lem:cal2}
For each $\rho\in \tilde{\Gamma}_{\mathbb{F}}$, we have
\[
                \Psi(P^\top P)_\rho = \sum_{1 \leq t \leq m_\rho} W_{\rho,t} W_{\rho,t}^\top
            \]
            where $W_{\rho,t} \in \mathbb{R}^{d_\rho \hat{n}\times (\dim\mathbb{F}) }$ is defined by
            \begin{align*}
                W_{\rho,t} =\sqrt{\frac{|\Gamma|}{d_\rho}}\cdot \begin{cases} \displaystyle
 \sum_{1 \leq l \leq d_\rho} \bm{e}_l \otimes \tilde{P}^{\top}Y^{\top}\be_{(\rho,t,l,1)} & (\text{if }\mathbb{F}=\mathbb{R})\\  
C\left(\displaystyle \sum_{1 \leq l \leq \frac{d_\rho}{2}}  
\bm{e}_l \otimes \left(\tilde{P}^\top Y^\top( \bm{e}_{(\rho,t,l,1)}+{\rm i} \bm{e}_{(\rho,t,l,2)})\right)\right) & (\text{if }\mathbb{F}=\mathbb{C}) \\ 
H\left( \displaystyle \sum_{1 \leq l \leq \frac{d_\rho}{4}} 
\bm{e}_l \otimes \left(\tilde{P}^\top Y^\top (\bm{e}_{(\rho,t,l,1)}+{\rm i}\bm{e}_{(\rho,t,l,2)} -{\rm j} \bm{e}_{(\rho,t,l,3)} +{\rm k} \bm{e}_{(\rho,t,l,4)})\right)\right) & (\text{if }\mathbb{F}=\mathbb{H}). \\
\end{cases}
            \end{align*}
\end{lemma}
    The proof of Lemma~\ref{lem:cal2} is in Appendix.
    For positive semidefinite complex Hermitian matrices, the following properties are well-known.
    A complex Hermitian matrix $X$ is positive semidefinite if and only if $C(X)$ is positive semidefinite.
    A positive semidefinite complex Hermitian matrices $X \in \mathbb{C}^{n \times n}$ with rank $r$ is written by $X=\sum_{1\leq i\leq r} v_i v_i^*$
    for some $v_1, \ldots, v_r \in \mathbb{C}^n$.
Also $\rank C(X) = 2\rank X$ for any complex matrix $X$.

Although $\mathbb{H}$ is not a field, the corresponding properties are known even for quaternionic matrices. 
For a set $v_1,\dots, v_k$ of vectors in $\mathbb{H}^n$, its (right) linear (in)dependence is defined in terms of the existence of scalars $\lambda_i\in \mathbb{H}$ such that $\sum_{i=1}^kv_i\lambda_i=\bm{0}$ with $\lambda_i\neq 0$ for some $i$.
Then the rank of a quoternionic matrix $X$ is defined by the maximum possible size of a linearly independent column vector set.
By \cite[Theorem~7.3]{zhang1997quaternions},  $\rank H(X) = 4\rank X$.
By \cite[Remark~6.1]{zhang1997quaternions},  $X$ is positive semidefinite if and only if $H(X)$ is positive semidefinite.
    Also by \cite[Corollary~6.2]{zhang1997quaternions}, a positive semidefinite $X \in \mathbb{H}^{n \times n}$ with rank $r$ is written by 
    $X=\sum_{1\leq i\leq r} v_i v_i^*$ for some $v_1, \ldots, v_r \in \mathbb{H}^n$.

    Now we are ready to prove Proposition~\ref{prop:range2}.

    \begin{proof}[Proof of Proposition~\ref{prop:range2}]
        By Lemma~\ref{lem:cal2}, for any $q \in {\cal C}_\theta(V)$, we have $\rank \Psi(Q^\top Q)_\rho \leq \dim \mathbb{F}\cdot m_\rho$ ($\rho \in \tilde{\Gamma}_\mathbb{F}$).
        Hence it suffices to show the other inclusion.
        Take any $X \in {\cal K}_{+,\Gamma}$ satisfying
        $\rank X_\rho \leq \dim \mathbb{F}\cdot m_\rho$ ($\rho \in \tilde{\Gamma}_\mathbb{F}$).
        We want to find $p \in {\cal C}_\theta(V)$ satisfying $\Psi(P^\top P)=X$.

The proof is identical to that for the case of absolutely irreducible groups given in  Proposition~\ref{prop:range}.
Indeed, if $\rho \in \tilde{\Gamma}_\mathbb{C}$ (resp., $\rho \in \tilde{\Gamma}_\mathbb{H}$), then $C^{-1}(X_\rho)$ (resp., $H^{-1}(X_\rho)$) is a positive semidefinite complex (resp., quaternionic) matrix with rank at most $m_\rho$.   
Hence, for each $\rho\in \tilde{\Gamma}_{\mathbb{F}}$, there exist vectors 
        $v_{\rho,1}, \ldots, v_{\rho,m_\rho}$ in $\mathbb{F}^\frac{d_\rho \hat{n}}{\dim \mathbb{F}}$ such that
\begin{equation}\label{eq:5.5}
X_\rho=\sum_{1 \leq t \leq m_\rho} v_{\rho,t} v_{\rho,t}^*, \quad
C^{-1}(X_\rho)=\sum_{1 \leq t \leq m_\rho} v_{\rho,t} v_{\rho,t}^*, \quad \text{ or } \quad 
H^{-1}(X_\rho)=\sum_{1 \leq t \leq m_\rho} v_{\rho,t} v_{\rho,t}^*
\end{equation}
depending on the type of $\rho$.    
Applying the same calculation as that in Proposition~\ref{prop:range}, one can find 
a $d \times \hat{n}$ matrix $\tilde{P}$ such that 
\begin{align*}
v_{\rho,t}&=\sqrt{\frac{|\Gamma|}{d_{\rho}}}\displaystyle \sum_{1 \leq l \leq \frac{d_\rho}{2}}  
\bm{e}_l \otimes \tilde{P}^\top Y^\top \bm{e}_{(\rho,t,l,1)} & (\text{if }\mathbb{F}=\mathbb{R}), \\ 
C(v_{\rho,t})&=\sqrt{\frac{|\Gamma|}{d_{\rho}}}C\left(\displaystyle \sum_{1 \leq l \leq \frac{d_\rho}{2}}  
\bm{e}_l \otimes \left(\tilde{P}^\top Y^\top( \bm{e}_{(\rho,t,l,1)}+{\rm i} \bm{e}_{(\rho,t,l,2)})\right)\right) & (\text{if }\mathbb{F}=\mathbb{C}), \\ 
H(v_{\rho,t})&=\sqrt{\frac{|\Gamma|}{d_{\rho}}}H\left( \displaystyle \sum_{1 \leq l \leq \frac{d_\rho}{4}} 
\bm{e}_l \otimes \left(\tilde{P}^\top Y^\top (\bm{e}_{(\rho,t,l,1)}+{\rm i}\bm{e}_{(\rho,t,l,2)} -{\rm j} \bm{e}_{(\rho,t,l,3)} +{\rm k} \bm{e}_{(\rho,t,l,4)})\right)\right) & (\text{if }\mathbb{F}=\mathbb{H}). 
\end{align*}
for all $\rho\in \tilde{\Gamma}$ and $1\leq t\leq m_{\rho}$.

%
Let $p$ be  a $d$-dimensional configuration defined by $P=\sum_{\gamma \in \Gamma} \bm{e}_\gamma^\top \otimes (\theta(\gamma)\tilde{P})$.
Then $p$ is compatible with $\theta$.
Also, by (\ref{eq:5.5}), the definition of $\tilde{P}$, and Lemma~\ref{lem:cal2},  we have $\Psi(P^\top P)_\rho=X_\rho$.
Since $P^{\top}P=\Psi^{-1}(X)\in {\cal L}^V$, we have $P^{\top}P \bm{1}_V$, meaning that $p(V)=\bm{0}$. 
Thus $p$ satisfies $p\in {\cal C}_{\theta}(V)$ and $\Psi(P^{\top}P)=X$ as required.
    \end{proof}

\subsection{Proof of Theorem~\ref{thm:sym} for General Case} \label{sec:5.3}
    We give the proof of Theorem~\ref{thm:sym} for general case.
    Let $\mathbb{Q}_{\theta, \Gamma}$ be the finite extension field of $\mathbb{Q}$ generated by $\mathbb{Q}$ and the entries of $\theta(\gamma)$ for $\gamma \in \Gamma$ and those of $\rho(\gamma)$ for $\rho\in \tilde{\Gamma}$ and $\gamma\in \Gamma$.
Let $\Psi$ be the orthogonal transformation given in Proposition~\ref{prop:str2}, and we consider the block-diagonalization of the SDP problem. The resulting SDP problem (P$^\Psi$) is as given in Section~\ref{sec:4.3}.
Proposition~\ref{prop:uniqueness} implies that (P$^\Psi$) has the unique solution $\Psi(P^{\top}P)$ if $(G,p)$ is universally rigid (c.f.~Proposition~\ref{prop:2}).

    To apply Gortler-Thurston's argument, we need to understand the facial structure of $\mathcal{K}_{+,\Gamma} = \Psi(({\cal L}^V)^\Gamma)$.
By (\ref{eq:K_G}), $\mathcal{K}_{+,\Gamma}$ is the direct sum of  positive semidefinite cones,  positive semidefinite Laplacian cones, 
cones of the form $\mathcal{C}^k\cap \mathcal{S}_+^k$, and cones of the form  $\mathcal{H}^k\cap \mathcal{S}_+^k$ for some integer $k$.

Through the map $C$, 
$\mathcal{C}^k\cap \mathcal{S}_+^k$ can be identified with 
the cone of positive semidefinite complex Hermitian matrices of size $k\times k$.
Similarly, through the map $H$, 
$\mathcal{H}^k\cap \mathcal{S}_+^k$ can be identified with 
the cone of positive semidefinite quaternionic Hermitian matrices of size $k\times k$.
%
Hence the following proposition can be proved by the identical manner as the proof of  Proposition~\ref{prop:psdface}.
    \begin{prop}\label{prop:chface}
        \begin{itemize}
        \item[(i)] For $A \in \mathcal{C}^k\cap \mathcal{S}_+^k$ with rank $2r$, $\dim F_{\mathcal{C}^k\cap \mathcal{S}_+^k}(A)=r^2$.

        For $B \in \mathcal{C}^k \cap \mathcal{S}^k$, the hyperplane $\{ \langle X,B \rangle=0 : X \in \mathcal{C}^k \cap \mathcal{S}^k\}$ in $\mathcal{C}^k \cap \mathcal{S}^k$
        exposes $F_{\mathcal{C}^k\cap \mathcal{S}_+^k}(A)$ if and only if $\rank A + \rank B = k$, $\langle A,B \rangle=0$, and $B \succeq 0$.

        \item[(ii)] For $A \in \mathcal{H}^k\cap \mathcal{S}_+^k$ with rank $4r$, $\dim F_{\mathcal{H}^k\cap \mathcal{S}_+^k}(A)=2r^2-r$.
        
        For $B \in \mathcal{H}^k \cap \mathcal{S}^k$, the hyperplane $\{ \langle X,B \rangle=0 : X \in \mathcal{H}^k \cap \mathcal{S}^k\}$ 
        in $\mathcal{H}^k \cap \mathcal{S}^k$ exposes $F_{\mathcal{H}^k\cap \mathcal{S}_+^k}(A)$ if and only if $\rank A + \rank B = k$, $\langle A,B \rangle=0$, and $B \succeq 0$.
        \end{itemize}
    \end{prop}

    Now we give the proof of Theorem~\ref{thm:sym} for general case.
    \begin{proof}[Proof of the necessity of Theorem~\ref{thm:sym}]
        Let $(G,p)$ be a $\theta$-symmetric framework with  a point group $\theta:\Gamma \rightarrow O(\mathbb{R}^d)$ which is generic modulo symmetry 
        and universally rigid.
        Suppose also that $p$ affinely spans $\mathbb{R}^d$. By a translation, we may suppose $p \in {\cal C}_\theta(V)$.
        We shall apply Proposition~\ref{prop:GT} to the ambient space $\mathcal{K}_\Gamma$ defined in (\ref{eq:K_G}) and a point $\Psi(P^\top P)$.
        To do so, we need to prove the local genericity of $\Psi(P^\top P)$.
        \begin{claim} \label{claim:2}
            For \[
            k=\sum_{\rho \in \tilde{\Gamma}_\mathbb{R}} \binom{m_\rho +1}{2} 
            + \sum_{\rho \in \tilde{\Gamma}_\mathbb{C}} m_\rho^2
            + \sum_{\rho \in \tilde{\Gamma}_\mathbb{H}} (2m_\rho^2-m_\rho),
            \]
            $\Psi(P^\top P)$ is locally generic over $\mathbb{Q}_{\theta,\Gamma}$ in $\text{ext}_k(\mathcal{K}_{+,\Gamma})$.
        \end{claim}
        \begin{proof}
            The facial structure of $\text{ext}_k(\mathcal{K}_{+,\Gamma})$ are described by Proposition~\ref{prop:psdface}, Proposition~\ref{prop:realface} and Proposition~\ref{prop:chface}.
            We also have Proposition~\ref{prop:range2}. 
            Hence the proof follows the same line as Claim~\ref{claim:1}.
        \end{proof}

        Define a linear subspace $\mathcal{K}_\Gamma(G)$ of ${\cal K}_\Gamma$ and a projection $\pi : {\cal K}_\Gamma \rightarrow {\cal K}_\Gamma (G)$ 
        similarly as in the previous section.
        Then by (the generalization of) Proposition~\ref{prop:GT}, there exists a hyperplane $H=\{\langle X,L \rangle =0 : X \in \mathcal{K}_\Gamma(G)\}$ in $\mathcal{K}_\Gamma(G)$ 
        defined by some $L \in \mathcal{K}_\Gamma(G)$ such that $\pi^{-1}(H)=\{\langle X,L \rangle =0 : X \in \mathcal{K}_\Gamma \}$ exposes $F_{\mathcal{K}_{+,\Gamma}}(\Psi(P^\top P))$.
        By Proposition~\ref{prop:psdface} ,Proposition~\ref{prop:realface} and Proposition~\ref{prop:chface}, we have
        \begin{align*}
            & \rank L_\rho =\begin{cases}
            \hat{n}-m_\tri-1 & (\rho =\tri) \\
            d_\rho\hat{n} - \dim \mathbb{F} \cdot m_\rho & (\rho \in \tilde{\Gamma}_\mathbb{F} \setminus \{1\}),
            \end{cases} \\
            &L_\rho \succeq 0,~ \langle L_\rho, \Psi(P^\top P)_\rho \rangle =0.
        \end{align*}
        By (\ref{eq:5-1}), (\ref{eq:5-2}), we have 
        \begin{align*}
            \rank L &= \hat{n}-m_\tri-1 + \sum_{\mathbb{F} = \mathbb{R}, \mathbb{C}, \mathbb{H}}  
            \sum_{\rho \in \tilde{\Gamma}_\mathbb{F} \setminus \{1\}}  \frac{d_\rho}{\dim \mathbb{F}}(d_\rho\hat{n} - \dim \mathbb{F} \cdot m_\rho) \\
            &=\hat{n}|\Gamma| - d -1 = n-d-1
        \end{align*}
        As $L\in \mathcal{K}_\Gamma(G)$, $\Psi^{-1}(L)$ is a weighted Laplacian on $G$ satisfying
        \[
            \Psi^{-1}(L) \succeq 0,~ \langle \Psi^{-1}(L),~ P^\top P \rangle =0, \rank \Psi^{-1}(L)=n-d-1.
        \]
        Therefore $\Psi^{-1}(L)$ satisfies the property of the statement.
This completes the proof.
    \end{proof}

\section*{Acknowledgments}
This work was supported by JST ERATO Grant Number JPMJER1903 
and JSPS KAKENHI Grant Number JP18K11155.
\bibliographystyle{plain}
\bibliography{myreference}
\appendix
\section{Proof of Lemma~\ref{lem:cal2}}
    In the proof of Lemma~\ref{lem:cal2}, we use the orthogonality relation of real irreducible representations (Proposition~\ref{prop:ortho}).
Recall that every $\rho \in \tilde{\Gamma}_\mathbb{R}$ 
(resp., $\rho\in \tilde{\Gamma}_{\mathcal{C}}, \rho\in \tilde{\Gamma}_\mathbb{H}$) is an orthogonal representation with $\rho(\gamma)\in \mathbb{R}^{d_{\rho}\times d_{\rho}}$ (resp., ${\cal C}^{d_\rho}$, ${\cal H}^{d_\rho}$) for $\gamma\in \Gamma$.
Fix the standard basis ${\cal B}_{\rho}$ of 
$\mathbb{R}^{d_{\rho}\times d_{\rho}}$, $\mathcal{C}^{d_{\rho}}$, $\mathcal{H}^{d_{\rho}}$ by
\begin{align*}
&\{E_{lm}:1\leq l,m \leq d_{\rho}\}, \\ 
&\{E_{lm} \otimes c(1), E_{lm} \otimes c(\mi) : 1\leq l,m\leq d_{\rho}/2 \}, \\
&\{E_{lm} \otimes h(1), E_{lm} \otimes h(\mi), E_{lm} \otimes h(\mj), E_{lm} \otimes h(\mk) : 1\leq l,m\leq d_{\rho}/4\},
\end{align*}
respectively.
For each $B \in B_\rho$, the {\em coordinate vector $\rho_B \in \mathbb{R}^{\Gamma}$ with respect to $B$ is defined by
\[ 
 \rho_B(\gamma) = \frac{1}{\dim \mathbb{F}} \langle \rho(\gamma), B \rangle \qquad (\gamma\in \Gamma).
\]
    Then, we have the following. 
    \begin{prop} \label{prop:ortho}
        $\left\{ \sqrt{ \frac{d_\rho}{|\Gamma|} } \rho_B : \rho \in \tilde{\Gamma}, B \in {\cal B}_\rho \right\}$ forms an orthogonal basis of $\mathbb{R}^{\Gamma}$.
    \end{prop}
    \begin{proof}
    By definition, for each $\rho \in \tilde{\Gamma}_\mathbb{F}$, $|{\cal B}_\rho|=\frac{{d_\rho}^2}{\dim \mathbb{F}}$.
    Hence by (\ref{eq:5-1}), $\sum_{\rho\in \tilde{\Gamma}}|{\cal B}_{\rho}|=|\Gamma|$.
Therefore, it suffices to prove that 
$\left\{ \sqrt{ \frac{d_\rho}{|\Gamma|} } \rho_B : \rho \in \tilde{\Gamma}, B \in B_\rho \right\}$ is an orthonormal set.
By Schur orthogonality (Proposition~\ref{prop:schur}) over $\mathbb{C}$, 
 for two inequivalent real irreducible representations $\rho, \rho' \in \tilde{\Gamma}$ and any $B \in B_\rho, B' \in B_{\rho'}$, 
$\rho_B$ and $\rho_{B'}$ are orthogonal to each other.
    Hence it suffices to prove that for each $\rho$, $\left\{\sqrt{\frac{d_\rho}{\Gamma}} \rho_B:B\in B_\rho \right\}$ forms an orthonormal set.

    When $\rho$ is of real type (i.e., $\rho$ is irreducible over $\mathbb{C}$, this follows from Schur orthogonality, again.

Suppose that $\rho$ is of complex type.
By Proposition~\ref{prop:type}, $\rho$ is (unitary) equivalent to $\pi\otimes \overline{\pi}$ for some complex irreducible representation $\pi$ of degree $\frac{d_{\rho}}{2}$, and by $\rho \in {\cal C}^{d_{\rho}}$ we may 
suppose  $\rho(\gamma)=\Re \pi(\gamma) \otimes c(1) + \Im \pi(\gamma) \otimes c(\mi)$
 for all $\gamma \in \Gamma$.
Then
\begin{equation}\label{eq:ortho1}
\rho_{E_{k,l}\otimes c(1)}=\Re \pi_{k,l}\text{ and }
\rho_{E_{k,l}\otimes c({\rm i})}=\Im \pi_{k,l},
\end{equation}
where $\pi_{k,l}$ denotes the vector in $\mathbb{C}^{\Gamma}$ such that 
$\pi_{k,l}(\gamma)$ is the $(k,l)$-entry of $\pi(\gamma)$.
  
For a complex number $c\in \mathbb{C}$, 
we have $\begin{pmatrix} c & \overline{c}\end{pmatrix}\begin{pmatrix} \frac{1}{2} & -\frac{\rm i}{2} \\ \frac{1}{2} & \frac{\rm i}{2}\end{pmatrix} =\begin{pmatrix} \Re c & \Im c\end{pmatrix}$.
Note that $\begin{pmatrix} \frac{1}{2} & -\frac{\rm i}{2} \\ \frac{1}{2} & \frac{\rm i}{2}\end{pmatrix}$ is unitary.
Applying this unitary transformation entry-wise to each pair $(\pi_{k,l},\overline{\pi_{k,l}})$,  
the set of vectors 
$\{\pi_{k,l}: 1\leq k,l\leq \frac{d_{\rho}}{2}\}\cup  
\{\overline{\pi_{k,l}}: 1\leq k,l\leq \frac{d_{\rho}}{2}\}$ 
is mapped to 
$\{\Re\pi_{k,l}: 1\leq k,l\leq \frac{d_{\rho}}{2}\}\cup \{\Im \pi_{k,l}: 1\leq k,l\leq \frac{d_{\rho}}{2}\}$.
Since this is a unitary transformation, the orthogonality of the former set (which follows from Schur orthogonality) implies that the orthogonality of the latter set. By (\ref{eq:ortho1}), this in turn implies the orthogonality of $\left\{\sqrt{\frac{d_\rho}{\Gamma}} \rho_B:B\in B_\rho \right\}$.


Finally, suppose that $\rho$ is of quanternionic type.
By Proposition~\ref{prop:type}, $\rho$ is
(complex-)equivalent to the direct sum of two copies of a self-conjugate irreducible representation $\pi$.
As $\rho\in {\cal H}^{d_{\rho}}$,
$\rho(\gamma)=a(\gamma)\otimes h(1) + b(\gamma) \otimes h(i)+c(\gamma) \otimes h(j) + d(\gamma) \otimes h(k)$ 
    for some $a,b,c,d:\Gamma \rightarrow M_{\frac{d_\rho}{4}}(\mathbb{R})$,
and we may take $\pi$ such that 
    \begin{align*}
    \pi(\gamma) &= a(\gamma) \otimes \begin{pmatrix} 1 & 0 \\ 0 & 1 \end{pmatrix} 
            + b(\gamma) \otimes \begin{pmatrix} i & 0 \\ 0 & -i \end{pmatrix}
            + c(\gamma) \otimes \begin{pmatrix} 0 & 1 \\ -1 & 0 \end{pmatrix}
            + d(\gamma) \otimes \begin{pmatrix} 0 & i \\ i & 0 \end{pmatrix}.
    \end{align*}
Hence, by using Schur orthogonality of $\pi$, one can check the statement by the same manner as the case of complex type.
    \end{proof}

    \begin{proof}[Proof of Lemma~\ref{lem:cal2}]
We apply the same calculation as that in the proof of Lemma~\ref{lem:1}.
Then, for each $\rho\in \tilde{\Gamma}$, we have
    \begin{equation} \label{eq:ap1}
        \Psi(P^\top P)_\rho = (I_{d_\rho}\otimes Y\tilde{P})^\top A_\rho (I_{d_\rho}\otimes Y\tilde{P}).
    \end{equation}
    with
    \begin{equation}\label{eq:ap2}
        A_\rho =\sum_{\gamma \in \Gamma} \rho(\gamma) \otimes 
 \bigoplus_{\rho' \in \tilde{\Gamma}} 
        \bigoplus_{1 \leq t \leq m_{\rho'}} \rho'(\gamma).
    \end{equation}
If  $\rho$ is of real type, the proof is identical to that of Lemma~\ref{lem:1}.

Suppose that $\rho$ is of complex type.
    By Proposition~\ref{prop:ortho}, for a real irreducible representation $\rho' \in \tilde{\Gamma}$, we have
    \[
        \sum_{\gamma \in \Gamma} \rho(\gamma) \otimes \rho'(\gamma) = 
        \begin{cases}
            \displaystyle\frac{|\Gamma|}{d_\rho} \sum_{1 \leq l,m \leq \frac{d_\rho}{2}} 
\left(E_{lm} \otimes c(1) \otimes E_{lm} \otimes c(1)+
E_{lm}\otimes c(\mi)\otimes E_{lm} \otimes c(\mi)\right)
            & (\text{if }\rho=\rho'), \\
            O & (\text{if }\rho \neq \rho').
        \end{cases}
    \]
To see how the sum in the above relation for $\rho=\rho'$ can be simplified,
let $\{\be_{(\ell,a)}: 1\leq \ell\leq \frac{d_{\rho}}{2}, 1\leq a\leq 2\}$ be the standard basis of $\mathbb{R}^{d_{\rho}\times d_{\rho}}$.
Then 
\begin{align*}
&\sum_{1 \leq l,m \leq \frac{d_\rho}{2}} 
\left(E_{lm} \otimes c(1) \otimes E_{lm} \otimes c(1)+
E_{lm}\otimes c(\mi)\otimes E_{lm} \otimes c(\mi)\right) \\
&= \sum_{1 \leq l,m \leq \frac{d_\rho}{2}} E_{lm} \otimes \begin{pmatrix}
                E_{lm} \otimes c(1) & -E_{lm} \otimes c(\mi) \\ 
                E_{lm} \otimes c(\mi) & E_{lm} \otimes c(1)
            \end{pmatrix} \\
&=\sum_{1 \leq l,m \leq \frac{d_\rho}{2}} \be_l\be_m^{\top}\otimes \begin{pmatrix}
                \be_l\be_m^{\top} \otimes c(1)c(1)^{\top} & \be_l\be_m^{\top} \otimes c(1)c(\mi)^{\top} \\ 
                \be_l\be_m^{\top} \otimes c(\mi)c(1)^{\top} & \be_l\be_m^{\top} \otimes c(\mi)c(\mi)^{\top}
            \end{pmatrix} \\
        &=
        \left( \sum_{1 \leq l \leq \frac{d_\rho}{2}} \bm{e}_l \otimes \begin{pmatrix} \bm{e}_l \otimes c(1) \\ \bm{e}_l \otimes c(\mi) \end{pmatrix} \right)
        \left( \sum_{1 \leq l \leq \frac{d_\rho}{2}} \bm{e}_l \otimes \begin{pmatrix} \bm{e}_l \otimes c(1) \\ \bm{e}_l \otimes c(\mi) \end{pmatrix} \right)^\top \\
&=\left( \sum_{1 \leq l \leq \frac{d_\rho}{2}} \bm{e}_l \otimes \begin{pmatrix} \bm{e}_{(l,1)} & \bm{e}_{(l,2)}\\ \bm{e}_{(l,2)} & -\bm{e}_{(l,1)} \end{pmatrix} \right)
        \left( \sum_{1 \leq l \leq \frac{d_\rho}{2}} \bm{e}_l \otimes \begin{pmatrix} \bm{e}_{(l,1)} & \bm{e}_{(l,2)}\\ \bm{e}_{(l,2)} & -\bm{e}_{(l,1)} \end{pmatrix} \right)^\top \\
&=\left( \sum_{1 \leq l \leq \frac{d_\rho}{2}} \bm{e}_l \otimes \begin{pmatrix} \bm{e}_{(l,1)} & -\bm{e}_{(l,2)}\\ \bm{e}_{(l,2)} & \bm{e}_{(l,1)} \end{pmatrix} \right)
        \left( \sum_{1 \leq l \leq \frac{d_\rho}{2}} \bm{e}_l \otimes \begin{pmatrix} \bm{e}_{(l,1)} & -\bm{e}_{(l,2)}\\ \bm{e}_{(l,2)} & \bm{e}_{(l,1)} \end{pmatrix} \right)^\top.
    \end{align*}
Combining it with (\ref{eq:ap2}),  we have
    \begin{equation} \label{eq:ap3}
        A_\rho = \frac{|\Gamma|}{d_\rho} \sum_{1 \leq t \leq m_\rho} U_{\rho,t}U_{\rho,t}^\top
    \end{equation}
    with
    \[
        U_{\rho,t}= \sum_{1 \leq l \leq \frac{d_\rho}{2}} \bm{e}_l \otimes
        \begin{pmatrix} \bm{e}_{(\rho,t,l,1)} & -\bm{e}_{(\rho,t,l,2)}\\ \bm{e}_{(\rho,t,l,2)} & \bm{e}_{(\rho,t,l,1)} \end{pmatrix}.
    \]
Hence by (\ref{eq:ap1}) and (\ref{eq:ap3}), we have
$\Psi(P^{\top}P)_{\rho}=\sum_{1\leq t\leq m_{\rho}}W_{\rho,t}W_{\rho,t}^{\top}$ with 
\begin{align*}
W_{\rho,t}&=\sqrt{\frac{|\Gamma|}{d_{\rho}}}\sum_{1 \leq l \leq \frac{d_\rho}{2}} \bm{e}_l \otimes
        \begin{pmatrix} \tilde{P}^{\top}Y^{\top}\bm{e}_{(\rho,t,l,1)} & -\tilde{P}^{\top}Y^{\top}\bm{e}_{(\rho,t,l,2)}\\ \tilde{P}^{\top}Y^{\top}\bm{e}_{(\rho,t,l,2)} & \tilde{P}^{\top}Y^{\top}\bm{e}_{(\rho,t,l,1)} \end{pmatrix} \\ 
&=\sqrt{\frac{|\Gamma|}{d_{\rho}}}\sum_{1 \leq l \leq \frac{d_\rho}{2}} 
C\left( 
\bm{e}_l \otimes \left(\tilde{P}^\top Y^\top( \bm{e}_{(\rho,t,l,1)}+{\rm i} \bm{e}_{(\rho,t,l,2)})\right)\right)
\end{align*}
as required.

Finally suppose that  $\rho$ is of quaternionic type.
    For a real irreducible representation $\rho' \in \tilde{\Gamma}$, by Proposition~\ref{prop:ortho}, 
    $\sum_{\gamma \in \Gamma} \rho(\gamma) \otimes \rho'(\gamma)$ is $O$ if $\rho$ and $\rho'$ are not equivalent,
    and if $\rho = \rho'$, this is equal to 
    \[
        \frac{|\Gamma|}{d_\rho} \sum_{1 \leq l,m \leq \frac{d_\rho}{4}} E_{lm} \otimes \begin{pmatrix}
             E_{lm} \otimes h(1) & -E_{lm} \otimes h(\mi) & E_{lm} \otimes h(\mj) & -E_{lm} \otimes h(\mk) \\ 
             E_{lm} \otimes h(\mi) & E_{lm} \otimes h(1) & E_{lm} \otimes h(\mk) & E_{lm} \otimes h(\mj) \\
            -E_{lm} \otimes h(\mj) & -E_{lm} \otimes h(\mk) & E_{lm} \otimes h(1) & E_{lm} \otimes h(\mi) \\
             E_{lm} \otimes h(\mk) & -E_{lm} \otimes h(\mj) &-E_{lm} \otimes h(\mi) & E_{lm} \otimes h(1)
        \end{pmatrix}.
    \]
    For $\rho=\rho'$, this value is written as $\frac{|\Gamma|}{d_\rho} B B^\top$ with
    \[
    B = \sum_{1 \leq l \leq \frac{d_\rho}{4}} \bm{e}_l \otimes \begin{pmatrix} \bm{e}_l\otimes h(1) \\ \bm{e}_l\otimes h(\mi) 
    \\ -\bm{e}_l\otimes h(\mj) \\ \bm{e}_l\otimes h(\mk) \end{pmatrix}.
    \]
Applying the same calculation as that of the complex case, one can derive the desired form.
%
\end{proof}

\end{document}